\documentclass[12pt]{amsart}
\usepackage{bm}
\usepackage[dvipsnames,svgnames,x11names]{xcolor}
\usepackage[hyphens]{url}
\usepackage[colorlinks=true,linkcolor=Maroon,
citecolor=blue,urlcolor=blue,hypertexnames=false,linktocpage]{hyperref}
\usepackage{bookmark}
\usepackage{amsmath,thmtools,mathtools}
\mathtoolsset{showonlyrefs=true}

\usepackage{blindtext}

\newcounter{mgncount}

\usepackage{tikz}

\usepackage{fancyhdr}
\usepackage{esint}
\usepackage{enumerate}
\usepackage{pictexwd,dcpic}
\usepackage{graphicx}
\usepackage{caption}
\allowdisplaybreaks

\usepackage{graphicx}

\newtheorem*{thm-}{Theorem}
\declaretheorem[name=Theorem,numberwithin=section]{thm}
\declaretheorem[name=Remark,style=remark,sibling=thm]{rem}
\declaretheorem[name=Lemma,sibling=thm]{lemma}
\declaretheorem[name=Proposition,sibling=thm]{prop}
\declaretheorem[name=Definition,style=definition,sibling=thm]{defn}

\declaretheorem[name=Question,numbered=no]{question}

\numberwithin{equation}{section}

\newcommand{\ti}{\tilde}

\newcommand{\sub}{\subset}

\newcommand{\bbR}{\mathbb{R}}

\newcommand{\bbS}{\mathbb{S}}

\newcommand{\al}{\alpha}

\newcommand{\ga}{\gamma}
\newcommand{\de}{\delta}

\newcommand{\la}{\lambda}

\newcommand{\si}{\sigma}

\newcommand{\Om}{\Omega}

\newcommand{\La}{\Lambda}
\newcommand{\Ph}{\Phi}

\newcommand{\cA}{\mathcal{A}}
\newcommand{\cB}{\mathcal{B}}
\newcommand{\cC}{\mathcal{C}}

\newcommand{\cF}{\mathcal{F}}

\newcommand{\cH}{\mathcal{H}}

\newcommand{\cK}{\mathcal{K}}

\newcommand{\n}{\nabla}

\newcommand{\fr}[2]{\frac{#1}{#2}}

\renewcommand{\(}{\left(}
\renewcommand{\)}{\right)}

\newcommand{\eq}[1]{\begin{equation}\begin{alignedat}{2} #1 \end{alignedat}\end{equation}}

\newcommand{\br}[1]{\left(#1\right)}
\newcommand{\abs}[1]{\lvert #1\rvert}

\newcommand{\Ra}{\Rightarrow}

\newcommand{\mrm}{\mathrm}

\newcommand{\q}{\quad}

\makeatletter
\@namedef{subjclassname@2020}{%
	\textup{2020} Mathematics Subject Classification}
\makeatother

\begin{document}
	\title[Capillary Christoffel-Minkowski problem]
	{Capillary Christoffel-Minkowski problem}
	\author[Y. Hu, M. N. Ivaki, J. Scheuer]{Yingxiang Hu, Mohammad N. Ivaki,  Julian Scheuer}
	
\begin{abstract}
The result of Guan and Ma (Invent. Math. 151 (2003)) states that if $\phi^{-1/k} : \bbS^n \to (0,\infty)$ is spherically convex, then $\phi$ arises as the $\sigma_k$ curvature (the $k$-th elementary symmetric function of the principal radii of curvature) of a strictly convex hypersurface. In this paper, we establish an analogous result in the capillary setting in the half-space for $\theta\in (0,\pi/2)$: if $\phi^{-1/k} : \cC_{\theta} \to (0,\infty)$ is a capillary function and spherically convex, then $\phi$ is the $\sigma_k$ curvature of a strictly convex capillary hypersurface.
\end{abstract}
	
\maketitle

\section{Introduction}
The Christoffel-Minkowski problem seeks the necessary and sufficient conditions on a Borel measure $\mu$ defined on the unit sphere $\mathbb{S}^n \subset \mathbb{R}^{n+1}$ for it to represent the $j$-th surface area measure of a convex body $K$ (i.e. a compact, convex set with non-empty interior).

For $j=n$, the problem is known as the Minkowski problem, and a complete solution was provided by Minkowski \cite{Min97, Min03}, Aleksandrov 
\cite{Ale56}, 
Nirenberg \cite{Nir57}, Pogorelov \cite{Pog52,Pog71}, and Cheng-Yau \cite{CY76}. They showed that a Borel measure $\mu$, whose support is not contained within a closed hemisphere, is the surface area measure of a convex body if and only if  
\eq{
\int_{\mathbb{S}^n} u \, d\mu(u) = 0.
}
Moreover, the solution is unique up to translations. For further details and generalizations, we refer the reader to \cite{Sch14}, \cite{Lut93, LO95, CW06, BLYZ13, HLYZ16, Li19, BBCY19, HXY21, GLW22, LXYZ24}, and the recent survey \cite{HYZ25}. Parabolic approaches are discussed, for example, in \cite{CW00, BIS19, LWW20, CL21, BG23}.

When $j=1$, the problem becomes the Christoffel problem, which was completely solved by Firey \cite{Fir67} and Berg \cite{Ber69} using methods based on Green's function and subharmonic functions.

For the intermediate case $1<j<n$, where the necessary and sufficient conditions have not been fully established, the problem remains largely open within the class of Borel measures. However, significant progress has been made in the case where $\mu$ has a smooth density with respect to the spherical Lebesgue measure. Notable results include those of Firey \cite{Fir70} for the rotationally symmetric case and Guan and Ma \cite{GM03} with Sheng, Trudinger, and Wang \cite{STW04} in the general case. In particular, in \cite{GM03}, the authors proved a constant rank theorem that guarantees the strict convexity of the solutions. Additional developments, including the extension within the $L_p$-Brunn-Minkowski theory and other ambient spaces, are provided in \cite{HMS04, GX18}, \cite{BIS21, HLX24, LW24}, and the problem was revisited in \cite{Iva19, BIS23b, Zha24} using parabolic approaches.

The geometry of capillary hypersurfaces, i.e. those which meet other hypersurfaces subject to a constant intersection angle, has received increasing attention during the past years. Many results, which are by now classical for closed hypersurfaces, were recently generalized to the capillary setting, such as the classification of stable CMC surfaces in a ball \cite{WX19}, or the proof of Alexandrov-Fenchel inequalities for capillary hypersurfaces in the half space \cite{MWWX24}, see also \cite{HWYZ24}. Also, curvature flows have been investigated in this setting, such as locally constrained mean curvature flows \cite{MW23,SW24,WW20} or flows in the Lorentzian setting \cite{KLS25}. The capillary Minkowski problem was recently studied in \cite{MWW23}.

In this paper, we investigate the existence of strictly convex solutions to the capillary Christoffel-Minkowski problem, which will be discussed below.

Let $\{E_i\}_{i=1}^{n+1}$ be the standard orthonormal basis of $\bbR^{n+1}$.
Let $\Sigma$ be a properly embedded, smooth, compact, connected and orientable hypersurface in $\overline{\mathbb{R}^{n+1}_+}$ with $\operatorname{int}(\Sigma)\subset \mathbb{R}^{n+1}_+$ and boundary $\partial \Sigma \subset \partial \mathbb{R}^{n+1}_+$. We say $\Sigma$ is a capillary hypersurface with the constant contact angle $\theta\in (0,\pi)$ if the following condition on $\partial\Sigma$ is met:
\begin{equation}
\langle \nu, E_{n+1} \rangle \equiv \cos\theta.
\end{equation}
Here, $\nu$ denotes the outer unit normal of $\Sigma$. The enclosed region by the hypersurface $\Sigma$ and the hyperplane $\partial \mathbb{R}^{n+1}_+$ is denoted by $\widehat{\Sigma}$.

We say $\Sigma$ is strictly convex if $\widehat{\Sigma}$ is a convex body and the second fundamental form of $\Sigma$ is positive-definite (which implies that $\Sigma$ does not contain a line segment). Throughout this paper, we work with strictly convex capillary hypersurfaces.

The capillary spherical cap of radius $r$ and intersecting with $\partial \mathbb{R}^{n+1}_+$ at a constant angle $\theta \in (0, \pi)$ is defined as
\eq{
\mathcal{C}_{\theta, r} := \left\{ x \in \overline{\mathbb{R}_+^{n+1}} \mid |x + r \cos \theta E_{n+1}| = r \right\}.
}
For simplicity, we put $\mathcal{C}_\theta = \mathcal{C}_{\theta, 1}$.

Let $\tilde{\nu}=\nu-\cos\theta E_{n+1}:\Sigma\to \partial \cC_{\theta}$ denote the capillary Gauss map of $\Sigma$, which is a diffeomorphism, see \autoref{lem:nudiffeo}.  The capillary support function of $\Sigma$, $s=s_{\Sigma}: \cC_{\theta}\to \bbR$, is defined as
\eq{
s(\zeta) =\langle \tilde{\nu}^{-1}(\zeta), \zeta + \cos\theta E_{n+1} \rangle,\q \zeta \in \cC_{\theta}.
}
For example, the capillary support function of $\cC_{\theta}$ is given by
\eq{
\ell(\zeta)=\sin^2\theta-\cos\theta \langle \zeta,E_{n+1}\rangle.
}

The capillary area measure of $\Sigma$ is defined as
\eq{
m_{\theta}(\cB) := \int_{\tilde{\nu}^{-1}(\cB)} \left( 1 - \cos\theta \langle \nu, E_{n+1} \rangle \right) d\cH^n,
}
for all Borel sets $\cB$ of $\cC_{\theta}$. Here, $\cH^n$ denotes the $n$-dimensional Hausdorff measure. For a smooth, strictly convex capillary hypersurface $\Sigma$ we have
\eq{
dm_{\theta} = \frac{\ell}{\mathcal{K}\circ\tilde{\nu}^{-1}} d\sigma,
}
where $d\sigma$ represents the area element of $\mathcal{C}_\theta$ and $\cK$ is the Gauss curvature of $\Sigma$.

\begin{question}[Capillary Minkowski problem]
Given a measure $m$ on the spherical cap $\mathcal{C}_\theta$, is there a strictly convex capillary hypersurface whose capillary area measure is $m$?	
\end{question}

For the regular case of this question, i.e., when $dm = \frac{\ell}{f}\, d\sigma$ and $f\in C^2(\cC_{\theta})$ is positive, the question is equivalent to finding a smooth, strictly convex capillary hypersurface whose Gauss curvature is $f$. The following existence result was proved in \cite{MWW23}.

\begin{thm}\label{Min P}
Let $\theta \in (0, \frac{\pi}{2}]$. Suppose $0 < \phi \in C^2(\cC_\theta)$ satisfies  
\begin{equation}
\int_{\cC_\theta} \langle \zeta, E_{i} \rangle \phi(\zeta) \, d\sigma = 0, \quad i = 1,2,\ldots,n,  
\end{equation}  
where $\{E_{i}\}_{i=1,\ldots,n}$ is the horizontal basis of $\partial \mathbb{R}^{n+1}_+$.  Then there exists a $C^{3,\alpha}$ strictly convex capillary hypersurface $\Sigma \subset \overline{\mathbb{R}^{n+1}_{+}}$ such that its Gauss-Kronecker curvature $\mathcal{K}$ satisfies  
\begin{equation} \label{capillary-Minkowski-problem}
\frac{1}{\mathcal{K}(\tilde{\nu}^{-1}(\zeta))} = \phi(\zeta), \quad \forall \zeta \in \cC_\theta.
\end{equation}  
Moreover, $\Sigma$ is unique up to a horizontal translation in $\overline{\mathbb{R}^{n+1}_{+}}$.  
\end{thm}

Let $\bar{g}, \bar \nabla$ denote the standard metric and Levi-Civita connection on $\cC_{\theta}$ and put  $\tau[s] = \bar\n^2 s + s\bar g$.
The regular capillary Christoffel-Minkowski problem in the halfspace asks the following question.

\begin{question}Let $1\leq k<n$ and $0<\phi\in C^{\infty}(\cC_{\theta})$. Find the necessary and sufficient assumptions on $\phi$ so that there exists a strictly convex capillary hypersurface $\Sigma \subset \overline{\bbR^{n+1}_{+}}$ with the capillary support function $s$ satisfying
\eq{
\sigma_k(\tau[s],\bar{g}) = \phi \quad \text{on} \ \mathcal{C}_\theta.
}
\end{question}

Our main contributions are captured in the following theorems.

\begin{thm}\label{CM-thm}
Let $\theta \in (0,\frac{\pi}{2})$ and $1 \leq k < n$. Suppose $0 < \phi \in C^{\infty}(\cC_\theta)$ satisfies:
\begin{enumerate}
\item $\phi(-\zeta_1,\ldots, -\zeta_n,\zeta_{n+1}) = \phi(\zeta_1,\ldots,\zeta_n, \zeta_{n+1})$ for all $\zeta \in \cC_{\theta}$,
\item $\bar \nabla^2 \phi^{-\frac{1}{k}} + \bar g \, \phi^{-\frac{1}{k}} \geq 0$, and $\bar\nabla_{\mu} \phi^{-\frac{1}{k}} \leq \cot\theta \, \phi^{-\frac{1}{k}}$.
\end{enumerate}

Then there exists a smooth, strictly convex capillary hypersurface $\Sigma \subset \overline{\bbR^{n+1}_{+}}$ whose $k$-th elementary symmetric function of the principal radii of curvature is given by:
\eq{\label{capillary-CM-problem}
\sigma_k(\tilde{\nu}^{-1}(\zeta)) = \phi(\zeta).
}

Moreover, the capillary support function $s$ of $\Sigma$ satisfies
\eq{
s(-\zeta_1,\ldots, -\zeta_n,\zeta_{n+1}) = s(\zeta_1,\ldots,\zeta_n, \zeta_{n+1}), \quad \forall \zeta \in \cC_{\theta},
}
and this solution is unique within this class.\end{thm}

The following theorem can be considered as the exact analogues of the result of Guan and Ma \cite{GM03} in the capillary setting.

\begin{thm}\label{CM-thm1}
Let $\theta \in (0, \frac{\pi}{2})$ and $1 \leq k < n$. Suppose $0 < \phi \in C^{\infty}(\cC_\theta)$ satisfies:
\begin{enumerate}
\item $\int_{\cC_\theta} \langle \zeta, E_i \rangle \, \phi(\zeta) \, d\sigma = 0$ for all $i = 1, 2, \ldots, n$,
\item $\bar \nabla^2 \phi^{-\frac{1}{k}} + \bar g \, \phi^{-\frac{1}{k}} \geq 0$ and $\bar\nabla_{\mu} \phi^{-\frac{1}{k}} = \cot\theta \, \phi^{-\frac{1}{k}}$.
\end{enumerate}

Then there exists a smooth, strictly convex capillary hypersurface $\Sigma \subset \overline{\bbR^{n+1}_{+}}$ whose $k$-th elementary symmetric function of the principal radii of curvature is given by:
\eq{
\sigma_k(\tilde{\nu}^{-1}(\zeta)) = \phi(\zeta), \quad \forall \zeta \in \cC_\theta.
}

Moreover, $\Sigma$ is unique up to a horizontal translation in $\overline{\mathbb{R}^{n+1}_{+}}$.
\end{thm}

To prove \autoref{CM-thm1}, a main challenge is the construction of a suitable homotopy path. An interesting corollary of \autoref{CM-thm} is that there is a unique strictly convex capillary hypersurface whose $\sigma_k$ of radii of curvature is $\ell^{-k}$. This allows us to construct a homotopy path (cf. \autoref{homotopy path}) by suitably perturbing a linear interpolation between $\phi^{-1/k}$ and $\ell$, while ensuring that the path remains within an admissible class of functions satisfying (1) and (2). The other challenges, such as establishing the $C^2$ estimates, are overcome through surprisingly short arguments that rely on the specific structure of the problem and may offer a streamlined approach in related capillary curvature problems. Moreover, to prove that the solution is strictly convex, we follow the approach in \cite{BIS23}, which is a different route compared to the inductive procedure used in the work of Guan and Ma \cite{GM03}. We show that the capillary support function is uniformly bounded away from zero along the homotopy path, and that the smallest principal radii of curvature satisfies a viscosity inequality. In particular, it cannot vanish identically, since it is strictly positive at the point where the capillary support function attains its minimum.

We mention that our arguments here also prove the result of \cite{MWW23} via completely different arguments for obtaining $C^0$, $C^1$, $C^2$ estimates when $\phi \in C^{\infty}(\cC_{\theta})$ and $\theta \in (0,\pi/2)$.

\section{Capillary hypersurfaces in the halfspace}\label{sec:capillary}
In this section, we recall some basic facts about strictly convex capillary hypersurfaces $\Sigma\sub\overline{\mathbb{R}^{n+1}_+}$. 

\begin{lemma} \cite{MWWX24}\label{lem:nudiffeo}
The Gauss map $\nu: \Sigma \to \mathbb{S}^n$ has its image in
\eq{
\mathbb{S}^n_\theta := \{y \in \mathbb{S}^n \mid y_{n+1} \geq \cos \theta \},
}
and $\nu: \Sigma \to \mathbb{S}^n_\theta$ is a diffeomorphism.
\end{lemma}
A simple consequence of this lemma is that the hypersurface $\Sigma$ can be parametrized via the inverse capillary Gauss map:
\eq{
X&: \mathcal{C}_\theta \to \Sigma\\
X(\zeta) &= \tilde{\nu}^{-1}(\zeta) = \nu^{-1}(\zeta + \cos\theta E_{n+1}).
}

The capillary support function of $\Sigma$, $s=s_{\Sigma}: \cC_{\theta}\to \bbR$, is defined as
\eq{
s(\zeta) := \langle X(\zeta), \nu(X(\zeta)) \rangle = \langle \tilde{\nu}^{-1}(\zeta), \zeta + \cos\theta E_{n+1} \rangle,\q \zeta \in \cC_{\theta}.
}
Note that the word ``capillary" here emphasizes that the domain of $s$ is not $\bbS^n_{\theta}$ but $ \mathcal{C}_{\theta}$. Often, it is more convenient to work with $\bbS_{\theta}^n$ as the domain of $s$. To avoid confusion, we write $\hat{s}=\hat{s}_{\widehat{\Sigma}}: \bbS^n\to \bbR$ for the (standard) support function of the convex body $\widehat{\Sigma}$, which is defined as
\eq{
\hat{s}(u):=\max_{x\in \widehat{\Sigma}} \langle u,x\rangle,\q \forall u\in \bbS^n.
}

Let $(D, \delta)$ and $(\bar{\nabla}, \bar{g})$ be the standard connection and metric of Euclidean space $\mathbb{R}^{n+1}$ and $\mathbb{S}^n$, respectively. 

\begin{lemma}\label{cap gauss param} \cite[Lem. 2.4]{MWWX24}
    Assume that $\Sigma$ is a strictly convex capillary hypersurface in $\overline{\mathbb{R}^{n+1}_+}$. Then for the parametrization $X: \mathcal{C}_\theta \to \Sigma$, we have the following formulas:
    \begin{enumerate}
        \item $X(\zeta) = \bar{\nabla} s(\zeta) + s(\zeta) (\zeta+\cos\theta E_{n+1})$.
        \item $\bar{\nabla}_\mu s = \cot \theta s$ along $\partial \mathcal{C}_\theta$, where $\mu$ is the outward-pointing unit conormal to $\partial \mathcal{C}_\theta \subset \mathcal{C}_\theta$.
        \item The principal radii of curvature of $\Sigma$ at $X(\zeta)$ are the eigenvalues of $\tau[s]=\bar{\nabla}^2 s + s \bar{g}$ with respect to the round metric $\bar{g}$.
    \end{enumerate}
\end{lemma}

\begin{defn}
For $\theta \in (0, \pi)$, $f \in C^2(\mathcal{C}_\theta)$ is called a {\em capillary function} if it satisfies the following Robin-type boundary condition:
\eq{
\bar{\nabla}_\mu f = \cot \theta f \quad \text{along} \ \partial \mathcal{C}_\theta.
}
\end{defn}

In view of \autoref{cap gauss param}, $\sigma_n(\tau^{\sharp}[s])$ is the reciprocal of the Gauss curvature of $\Sigma$ at the point $X(\zeta)$ on $\Sigma$. Moreover $\zeta^i=\langle \zeta,E_i\rangle$ is a capillary function and from the divergence-free structure of $\sigma_k$, $\tau[\zeta^i] = 0$ for $i = 1, \ldots, n$ and integration by parts, it follows that
\eq{
\int_{\mathcal{C}_\theta} \sigma_k(\tau^{\sharp}[s](\zeta)) \zeta^i\, d\si= \frac 1k\int_{\cC_\theta}\sigma_k^{pq}\tau[\zeta^i]_{pq}s\,d\si=  0, \quad \text{for} \ i = 1, \ldots, n.
}

For a capillary function, the following crucial fact is known:

\begin{lemma}\cite[Lem. 2.7]{MWWX24} \label{key-lemma-capillary-function}
    If $f$ is a capillary function on $\mathcal{C}_\theta$, then for all $2 \leq \alpha \leq n$, 
    \eq{
    \bar{\nabla}^2 f(e_\alpha, \mu) = 0 \quad \text{along} \ \partial \mathcal{C}_\theta,
    }
    where $\{e_\alpha\}_{\alpha = 2, \ldots, n}$ is a local orthonormal frame on $\partial \mathcal{C}_\theta$.
\end{lemma}

Another crucial theorem that we need is the following statement.

\begin{thm}\label{constr hyper from s}Let $\theta \in (0, \frac{\pi}{2})$ and $s \in C^{\infty}(\cC_{\theta})$ be a capillary function with $\tau[s] > 0$. Then the image $D\hat{s}(\bbS^n_{\theta})$ is a smooth, strictly convex capillary hypersurface, where
\eq{
\hat{s}(\cdot) = s(\cdot - \cos \theta \, E_{n+1}).
}
\end{thm}
\begin{proof} 
Let $\Sigma = D\hat{s}(\bbS^n_\theta)$. It can be verified that $D\hat{s}: \bbS^n_\theta \to \bbR^{n+1}$ defines an immersion with positive Gauss curvature. Moreover, $D\hat{s}(\operatorname{int}(\bbS^n_\theta)) \subset \bbR^{n+1}_+$, $D\hat{s}(\partial \bbS^n_\theta) \subset \partial \bbR^{n+1}_+$, and $\Sigma$ intersects $\partial \bbR^{n+1}_+$ at the constant angle $\theta$; see \cite{MWWX24}.

For $u\in \bbS^n\cap E_{n+1}^{\bot}\sim \bbS^{n-1}$ define 
\eq{
h(u)=\frac{\hat{s}(\sin \theta u+\cos\theta E_{n+1})}{\sin \theta}.
} 
Note that $h$ is the support function of a closed, strictly convex hypersurface $\cH$ in $\partial \bbR^{n+1}_+$.
Since $\cH=Dh(\bbS^{n-1})=D\hat{s}(\partial \bbS^n_{\theta})=\partial \Sigma$, we see that $\partial \Sigma$ is embedded. Now from \cite[Thm. 3.5]{Gho02} it follows that $\Sigma$ bounds a convex body and thus it is a capillary hypersurface.
\end{proof}

\section{A Full rank theorem}

\begin{thm}[Full Rank Theorem]\label{thm: frt}
Let $\theta \in (0, \pi/2]$ and $1 \leq k < n$. Assume that $\bar \nabla^2 \phi^{-\frac{1}{k}} + \bar g \phi^{-\frac{1}{k}} \geq 0$, and that $\bar \nabla_{\mu} \log \phi \geq -k \cot \theta$. If $s$ is a capillary function with $\tau[s] \geq 0$ that satisfies $\sigma_k(\tau^{\sharp}[s]) = \phi$, then the smallest eigenvalue $\lambda_1$ of $\tau^{\sharp}$ cannot attain zero unless $\lambda_1 \equiv 0$ on $\cC_{\theta}$. In particular, if $s > 0$, we have $\lambda_1 > 0$ on $\cC_{\theta}$.
\end{thm}

\begin{proof}

Let $F=\sigma_k^{\frac{1}{k}}$ and $f=\phi^{\frac{1}{k}}$. If $\lambda_1 = 0$ somewhere in the interior of $\cC_{\theta}$, then we can argue as in \cite{BIS23} and use that $\lambda_1$ satisfies
\eq{
L[\lambda_1]:=F^{ij} \bar{\nabla}_{i,j}^2 \lambda_1- c (\lambda_1+ |\bar{\nabla} \lambda_1|)\leq 0
}
in a viscosity sense. Here, $c=c(\dot{F},\ddot{F},|\bar\nabla \tau|)$ is uniformly bounded in $\operatorname{int}(\cC_{\theta})$.  
By the strong maximum principle, $\lambda_1\equiv 0$ in $\cC_{\theta}$.

Hence, we may suppose $\lambda_1(p) = 0$ for some $p\in\partial \cC_{\theta}$ and $\lambda_1 > 0$ in the interior. Let $\{\mu\} \cup \{e_{\alpha}\}_{\alpha\geq 2}$ be an orthonormal basis of eigenvectors of $\tau$ at $p$, and set $e_1 = \mu$, where we used \autoref{key-lemma-capillary-function}. We first compute the required boundary derivatives: By  the Gaussian formula and Weingarten equation,
\eq{
0 = e_\beta(\tau(\mu,e_\alpha)) &= \bar \nabla_{\beta}\tau_{\alpha \mu}+\tau(\bar\nabla_{e_\beta}e_\alpha,\mu)+\tau(\bar\nabla_{e_\beta}\mu,e_\alpha)\\
 &=\bar \nabla_{\beta}\tau_{\alpha \mu}+\tau(-\cot\theta \bar g_{\al\beta}\mu,\mu)+\tau(\cot\theta e_\beta,e_\alpha)\\
 &=\bar \nabla_{\mu}\tau_{\alpha \beta}-\cot\theta\tau_{\mu\mu}\bar g_{\alpha\beta}+\cot\theta\tau_{\alpha\beta},
}
where we used that $\bar\nabla \tau$ is fully symmetric. Thus, we obtain
\eq{\label{important identity}
\bar \nabla_{\mu}\tau_{\alpha\beta}=(\tau_{\mu\mu}\bar g_{\alpha\beta}-\tau_{\alpha\beta})\cot\theta.
}
Next, we differentiate the equation and deduce 
\eq{\label{pf:FRT 1}
\bar\nabla_\mu \tau_{\mu\mu} = \frac{\bar\nabla_\mu f}{F^{\mu\mu}}-\sum_{\alpha}\frac{F^{\alpha\alpha}}{F^{\mu\mu}}\bar\nabla_\mu\tau_{\alpha\alpha} = \frac{\bar\nabla_\mu f}{F^{\mu\mu}}+\sum_{\alpha}\frac{F^{\alpha\alpha}}{F^{\mu\mu}}(\tau_{\alpha\alpha}-\tau_{\mu\mu})\cot\theta.
}

\emph{Step 1}: We prove 
\eq{\tau_{ii}(p)=0\q\Ra\q \bar\n_{\mu}\tau_{ii}(p)\geq 0.}
For $i=\alpha$, this follows immediately from \eqref{important identity}. For $i=\mu$, this follows after invoking \eqref{pf:FRT 1}, using the one-homogeneity of $F$, and the boundary condition for $f$.


\emph{Step 2}: Consider an interior ball $B_R(x_0)\sub\cC_\theta$ touching at $p$. Define an annular region $A_{R,\rho}=B_{R}(x_0)\setminus  \operatorname{int}(B_{\rho}(x_0))$ for some $0<\rho<R$.
For $x\in \cC_{\theta}$, let $r(x)=\operatorname{dist}(x,x_0)$ denote the distance of $x$ to $x_0$. We define
\eq{w(x)=e^{-\al R^2}-e^{-\al r(x)^2}.}
For any $x\in A_{R,\rho}$, the distance function $r(x)$ satisfies 
\eq{
\bar\nabla^2_{i,j}r(x)=\cot r(x) (\bar g_{ij}-\bar\nabla_i r \bar\nabla_j r).
}
Therefore,
\eq{
\bar \nabla_{i} w &=2\al r \bar\nabla_i r e^{-\al r^2},\\
\bar \nabla^2_{i,j} w&=-4\al^2 r^2 \bar\nabla_i r \bar \nabla_j r e^{-\al r^2}+2\al \bar\nabla_i r \bar\nabla_j r e^{-\al r^2}+2\al r \bar\nabla^2_{i,j} r e^{-\al r^2} \\
&=e^{-\al r^2}\( (-4\al^2 r^2+2\al-2\al r \cot r) \bar\nabla_i r \bar \nabla_j r+2\al r \cot r \bar g_{ij}\)
}
and
\eq{
L[w]= &~e^{-\al r^2}(-4\al^2 r^2+2\al-2\al r\cot r)|\bar\nabla r|_{\dot{F}}^2\\
 &~+2\al r\cot r e^{-\al r^2}\operatorname{tr}(\dot{F})-2 c \al r e^{-\al r^2}+c(e^{-\al r^2}-e^{-\al R^2}).
}
Assume that $\la \de_{ij} \leq F^{ij} \leq \La \de_{ij}$ in $B_R(x_0)$, where $0<\la\leq \La$ are some constants. Note that $0\leq r\cot r \leq 1$ for any $\rho \leq r\leq R$. Now take $\al>0$ sufficiently large to ensure in $A_{R,\rho}$, that 
\eq{\label{L of w}
L[w] \leq &~-e^{-\al r^2}\left[ (4\al^2r^2-2\al)\la -2\al n  \La +2c\al r-c\right]-ce^{-\al R^2}<0.
}

\emph{Step 3}: Since $\la_1(x)>0$ on $\partial B_\rho(x_0)$, there exists $\varepsilon>0$ such that 
\eq{
\psi(x):=\la_1(x)+\varepsilon w(x) > 0
}
on $\partial B_\rho(x_0)$. Note that $w = 0$ on $\partial B_R(x_0)$ and by our assumption there is no other point on $\partial B_R(x_0) \setminus \{p\}$ where $\la_1 = 0$, hence, $\psi > 0$ on $\partial B_R(x_0) \setminus \{p\}$. By the maximum principle for the viscosity supersolution $\psi$, we have $\psi\geq 0$ in the annulus, where $\psi$ is also the smallest eigenvalue of
\eq{
S_{ij}=\tau_{ij}+\varepsilon w \bar{g}_{ij}.
}
Suppose at $p$ the zero eigenvalue of $S(p)$ is attained in direction $e_{i}$, i.e. we also have $\tau_{ii} = 0$.
Let $\ga$ be a unit speed geodesic in direction $-\mu$ and $e_i$ be parallel transported along $\ga$. Then from step 1 we get
\eq{
{0\leq \frac{d}{dt}}_{|_{t=0^+}}S(\gamma(t))(e_i,e_i)=\bar\n_{-\mu}\tau_{ii}+\varepsilon \bar \nabla_{-\mu}w=-2\alpha \varepsilon R <0,
} 
a contradiction.

To complete the proof, note that at a point where $s$ attains its minimum, say $y$, we must have $\la_1(y) > 0$. In fact, $y$ cannot be on the boundary; otherwise, we would have $s(y) \cot\theta = \bar\nabla_{\mu}s(y) \leq 0$. Since $y \in \operatorname{int}(\cC_{\theta})$, $\tau[s]|_{y} \geq s(y)\bar{g} >0$.

\end{proof}

\section{A priori estimates}

\begin{lemma}\label{app: CW}Let $\theta \in (0, \pi/2]$ and suppose that $\Sigma$ is a strictly convex capillary hypersurface. Denote by $r_{\Omega}$ and $R_{\Omega}$ the (standard) inner and outer radii of the convex body $\Omega = \widehat{\Sigma} \cap E_{n+1}^{\perp} \subset E_{n+1}^{\perp} \cong \bbR^n$. Then
\eq{
\frac{R_{\Omega}^2}{r_{\Omega}} \leq C \max \lambda_n,
}
where $\lambda_n$ denotes the largest principal radius of curvature at each point, and the constant $C$ depends only on $n$ and $\theta$.
\end{lemma} 
\begin{proof}

By \cite[Prop. 2.4]{WWX24}, the second fundamental form of $\Sigma$ in $\mathbb{R}^{n+1}_+$, $\mathrm{II}=(h_{ij})$, and the second fundamental form of $\partial \Sigma$ in $\partial \mathbb{R}^{n+1}_+$, $\hat{\mathrm{II}}=(\hat{h}_{\alpha\beta})$, $2\leq \alpha,\beta\leq n$, are related by 
\begin{equation}
\mathrm{II}_{\alpha\beta}= \hat{\mathrm{II}}_{\alpha\beta} \sin\theta.
\end{equation}
Now, the claim follows from \cite[Lem. 2.2]{CW00}.

\end{proof}

\begin{lemma}\label{cap s bound to standard s bound}Let $\theta \in (0, \pi/2]$ and suppose that $\Sigma$ is a strictly convex capillary hypersurface. If $\lvert s_{\Sigma} \rvert \leq C$, then
\eq{
\lvert \hat{s}_{\widehat{\Sigma}} \rvert \leq \frac{C}{\sin \theta}.
}
\end{lemma}
\begin{proof}
Note that $\hat{s}(-E_{n+1})=0$ and $s(\zeta)=\hat{s}(\zeta +\cos \theta E_{n+1})$ for all $\zeta \in \cC_{\theta}$. Let $\Omega=\widehat{\Sigma}\cap E_{n+1}^{\bot}$, which we consider as a convex set in $ \bbR^{n+1}$, and write $\hat{s}_{\Omega}:\bbS^n\to \bbR$ for its the support function. For $u\in \bbS^n\setminus \operatorname{int}(\bbS^n_{\theta})$ with $u\neq -E_{n+1}$, in view of $\hat{s}_{\Omega}(x+tE_{n+1})=\hat{s}_{\Omega}(x)$ for all $t$, we have  
\eq{
    \hat{s}(u)=\hat{s}_{\Omega}(u)=\langle u,\tilde{u}\rangle \hat{s}_{\Omega}(\tilde{u}),
}
where $\ti u=\frac{u-\langle u,E_{n+1}\rangle E_{n+1}}{|u-\langle u,E_{n+1}\rangle E_{n+1}|}$. 
Hence,  
\eq{
    \hat{s}(u)=\langle u,\tilde{u}\rangle \hat{s}_{\Omega}(\tilde{u})=\frac{\langle u,\tilde{u}\rangle}{\sin \theta}\hat{s}(\sin \theta \ti u+\cos \theta E_{n+1})=\frac{\langle u,\tilde{u}\rangle}{\sin \theta}s(\sin \theta \ti u).
}

\end{proof}

\begin{lemma}\label{steiner} Let $\psi$ be one of the following functions $x^2$, $x^{p}$ for some $p<0$, or $-\log x$.  Suppose $\theta \in (0, \pi/2]$ and $\Sigma$ is a strictly convex capillary hypersurface. Then there exists a unique point $z \in \operatorname{int}(\widehat{\Sigma} \cap E_{n+1}^{\perp})$ such that
\begin{align}
\int_{\cC_{\theta}} \psi'\left(s_{\Sigma}(\zeta) - \langle \zeta, z \rangle\right) \zeta_i \, d\si = 0, \quad \forall i = 1, \ldots, n.
\end{align}
In particular, $s_{\Sigma - z} > 0$.\end{lemma}

\begin{proof}We adapt the proof from \cite{GN17} to the capillary setting. Note that for all $z \in \operatorname{int}(\widehat{\Sigma} \cap E_{n+1}^{\perp})$, it holds that $s_{\Sigma - z} > 0$. This can be seen by reflecting $\hat{\Sigma}$ across $E_{n+1}^{\perp}$, which yields a closed convex body in which $z$ lies in the interior.  Define  
\eq{
\cF(y)=\int_{\cC_\theta} \psi(s(\zeta)-\langle \zeta,y\rangle) \, d\si, \quad y\in \mrm{int}(\widehat{\Sigma}\cap E_{n+1}^\bot).
}
We will prove that $\mathcal{F}$ attains a minimum in the interior of $\Omega = \widehat{\Sigma} \cap E_{n+1}^{\perp}$. First, we note that $\mathcal{F}$ can be extended up to the boundary, as for every $y \in \partial \Omega$ the improper integral exists.

Let $y_0\in \partial \Omega=\partial \Sigma$ and suppose without loss of generality that $\cF(y_0)<\infty$.  We may translate and rotate $\Omega$ in $E_{n+1}^\bot$ so that $y_0$  lies at the origin and $E_1$ is the outer unit normal to $\partial\Omega=\partial \Sigma$, as a hypersurface of $\bbR^n$, at $y_0$. Thus, the outer unit normal of $\Sigma$ at $y_0$ is $x_0=\sin\theta E_1+\cos\theta E_{n+1}\in \partial \bbS^n_\theta$ (cf. \cite[(2.8)]{WWX24}) and $-tE_1\in \operatorname{int}(\Omega)$ for all sufficiently small $t>0$. Consequently, the function
\eq{
t \mapsto \mathcal{F}(-tE_1)
}
is $C^2$ on the interval $(0, \varepsilon)$ for sufficiently small $\varepsilon > 0$, with differentiation performed under the integral sign. We claim that, possibly after shrinking $\varepsilon$ further, the derivative with respect to $t$ is negative in this interval.

For $a^+=\sin\theta E_1\in \partial\cC_\theta$, we have  
\eq{
s(a^{+})&=\hat{s}(\sin\theta E_1+\cos\theta E_{n+1})=\max_{y\in \widehat{\Sigma}}\langle \sin\theta E_1+\cos\theta E_{n+1},y\rangle=0.
}
This implies that $\widehat{\Sigma}\subset \{ y\in \overline{\bbR^{n+1}_+}~|\sin\theta y_1+\cos\theta y_{n+1}\leq 0 \}$.  

Note that for $a^-=-\sin\theta E_1\in \partial\cC_\theta$, we must have $s(a^{-})>0$. Otherwise,  
\eq{
0\geq s(a^{-})&=s(-\sin\theta E_1)=\max_{y\in \widehat{\Sigma}}\langle -\sin\theta E_1+\cos\theta E_{n+1},y\rangle,
}
and thus $\widehat{\Sigma}\subset \{ y\in \overline{\bbR^{n+1}_+}~|-\sin\theta y_1+\cos\theta y_{n+1}\leq 0 \}$. However, this implies that  $V(\widehat{\Sigma})=0$.

Now, for any $\zeta^+=(\zeta_1,\zeta_2, \ldots, \zeta_{n+1}) \in \cC_{\theta}$ such that $\zeta_1> 0$, we define a corresponding point $\zeta^-$ by $\zeta^- = (-\zeta_1,\zeta_2, \ldots, \zeta_{n+1})$ and let $i(\zeta^+)$ be the point on $\Sigma$ such that $s(\zeta^+) = \langle \zeta^++\cos\theta E_{n+1}, i(\zeta^+) \rangle$. Note that the $\zeta_1$-component of $i(\zeta^+)$ is non-positive, since we have $\widehat{\Sigma}\subset \{ y\in \overline{\bbR^{n+1}_+}~|\sin\theta y_1+\cos\theta y_{n+1}\leq 0 \}$. Thus,  
\eq{ \label{zeta^->zeta^+}
s(\zeta^{-})=\hat{s}(\zeta^{-}+\cos\theta E_{n+1})&=\max_{y\in \widehat{\Sigma}}\langle \zeta^{-}+\cos\theta E_{n+1},y\rangle \\
& \geq \langle \zeta^{-}+\cos\theta E_{n+1},i(\zeta^+)\rangle\\
& \geq \langle \zeta^{+}+\cos\theta E_{n+1},i(\zeta^+)\rangle=s(\zeta^+).
}
Since $s(a^-)>s(a^+)$, by continuity,  the set of points with $s(\zeta^-)>s(\zeta^+)$ has a positive measure. 

We compute for $t>0$,
\eq{
\frac{d}{dt}\cF(-t E_1)&=\int_{\cC_\theta} \psi'(s(\zeta) +t\zeta_{1})\zeta_1 \, d\si\\
		&=\int_{\{\zeta\in \cC_\theta |\zeta_1>0\}}\br{\psi'(s(\zeta^+) + t\zeta_{1})-\psi'(s(\zeta^{-}) - t\zeta_{1})}\zeta_1 \, d\si.
}
Since $\psi'' > 0$, for $t_{i+1}, t_i \in (0,\epsilon) $ with $t_{i+1}\leq t_i$ we have
\eq{
\psi'(s(\zeta^-) - t_{i+1}\zeta_{1})-\psi'(s(\zeta^{+}) +t_{i+1}\zeta_{1})\geq \psi'(s(\zeta^-) - t_{i}\zeta_{1})-\psi'(s(\zeta^{+}) + t_{i}\zeta_{1}).
}
Moreover, since $s(\zeta)-t_1\langle \zeta,E_1\rangle>0$, for some constant $C$ we have
\eq{
\psi'(s(\zeta^-) - t_{1}\zeta_{1})-\psi'(s(\zeta^{+}) + t_{1}\zeta_{1}) & \geq -C.
}
Therefore, by the monotone convergence theorem:
\eq{
\lim_{t \to 0} \frac{d}{dt} \mathcal{F}(-tE_{1}) = \int_{\{\zeta \in \cC_\theta \mid \zeta_1 > 0\}} \left( \psi'(s(\zeta^+)) - \psi'(s(\zeta^{-})) \right) \zeta_1 \, d\sigma < 0.
}
That is, $t \mapsto \mathcal{F}(-tE_1)$ is decreasing for small $t > 0$. For $\psi(x)=x^2$, this shows $\cF$ does not attain its minimum on $\partial \Om$. We consider the other choices of $\psi$.

Note that for $t>0$,
\eq{\cF(-tE_1) &= \int_{\{\zeta_1\geq 0\}}\psi(s(\zeta) + t\zeta_1) + \int_{\{\zeta_1< 0\}}\psi(s(\zeta) + t\zeta_1). 
}
Since $\psi'<0$, again by the monotone convergence theorem, we find
\eq{
\cF(-tE_1)\to \cF(0)<\infty.
}
Therefore, $\mathcal{F}(-tE_1) < \mathcal{F}(0)$ for small $t > 0$.

Now suppose the minimum of $\mathcal{F}$ in $\Omega$ is attained at $z \in \operatorname{int}(\Omega)$, then 
\eq{
\left\langle D\mathcal{F}|_{z}, E_i \right\rangle = 0, \quad \forall i = 1, \ldots, n,
}
which is \eqref{entropy l2}. The uniqueness of $z$ follows from the strict convexity of the functional $\cF$:
\eq{
D^2\cF(E_i,E_j)=\int_{\cC_{\theta}}\psi''\zeta_i\zeta_j \, d\sigma, \q 1\leq i,j\leq n.
}
\end{proof}

\begin{thm}\label{steiner2} Let $\psi$ be one of the following functions $x^2$, $x^{p}$ for some $p<0$, or $-\log x$. Suppose $\theta \in (0, \pi/2]$ and $0 < h \in C^{3}(\cC_{\theta})$ is a capillary function with $\tau[h] \geq 0$. Then there is a convex body $\widehat{\Sigma} \subset \overline{\bbR^{n+1}_+}$ such that $\widehat{\Sigma} \cap E_{n+1}^{\perp}$ has non-empty interior relative to $E_{n+1}^{\perp}$ and $h(\zeta) = \hat{s}_{\widehat{\Sigma}}(\zeta + \cos\theta E_{n+1})$. Moreover, there is a unique point $z \in \operatorname{int}(\Omega)$ such that
\begin{align}\label{entropy l3}
\int_{\cC_{\theta}} \psi'\left( h(\zeta) - \langle \zeta, z \rangle \right) \zeta_i \, d\sigma = 0, \quad \forall i = 1, \ldots, n,
\end{align}
and $h(\zeta) - \langle \zeta, z \rangle > 0$ for all $\zeta \in \cC_{\theta}$.\end{thm}
\begin{proof}
	Define $s_t = (1 - t)\ell + t h$ for $t \in [0,1)$. By \autoref{constr hyper from s}, $s_t$ is the capillary support function of a strictly convex capillary hypersurface $\Sigma_t$. Moreover, as $t \to 1$, the convex bodies $\widehat{\Sigma}_t$ converge in the Hausdorff distance to a convex set $\widehat{\Sigma}_1 \subset \overline{\mathbb{R}_+^{n+1}}$. Let $\Omega_t = \widehat{\Sigma}_t \cap E_{n+1}^{\bot}$ for $t\in [0,1]$. Then $\Omega_t$ also converges in the Hausdorff distance to $\Omega_1$ as $t \to 1$.  We consider $\Omega_t$ as a convex set in $\bbR^n$ and write $\hat{s}_{\Omega_t}:\bbS^{n-1}\to \bbR$ for its support function.

We claim that  $\widehat{\Sigma}_1$ is a convex body and $h(\zeta)=\hat{s}_{\widehat{\Sigma}_1}(\zeta+\cos\theta E_{n+1})$ for all $\zeta \in \cC_{\theta}$. For $t\in [0,1)$, we have
	\eq{
	s_{\Sigma_t}(\zeta)&=\hat{s}_{\widehat{\Sigma}_t}(\zeta+\cos\theta E_{n+1}),\q \forall\zeta \in \cC_{\theta}\\
	\hat{s}_{\Omega_t}(\tilde{u})&=\frac{\hat{s}_{\widehat{\Sigma}_t}(\sin \theta \tilde{u} +\cos\theta E_{n+1})}{\sin \theta},\q  \forall  \ti u \in \bbS^{n}\cap E_{n+1}^{\bot}\sim \bbS^{n-1}.
	}
Hence, after taking the limit $t\to 1$, we obtain	
	\eq{\label{key observations}
	h(\zeta)&=\hat{s}_{\widehat{\Sigma}_1}(\zeta+\cos\theta E_{n+1}),\q \forall\zeta \in \cC_{\theta}\\
	\hat{s}_{\Omega_1}(\tilde{u})&=\frac{\hat{s}_{\widehat{\Sigma}_1}(\sin \theta \tilde{u} +\cos\theta E_{n+1})}{\sin \theta},\q \forall \ti u \in  \bbS^{n-1}.
	}
Therefore $\hat{s}_{\Omega_1}>0$ in $ \bbS^{n-1}$, that is, $\Omega_1$ has non-empty interior relative to $E_{n+1}^{\bot}$.
 
Define
 \eq{
\cF(y)=\int_{\cC_\theta} \psi(h(\zeta)-\langle \zeta,y\rangle) \, d\si, \quad y\in \widehat{\Sigma}_1\cap E_{n+1}^\bot.
}
Let $y_0 \in \partial \Omega_1$. We may translate $\Omega_1$ so that $y_0=0$. By \cite[Thm. 1.12]{Bus12}, there are rectangular
coordinates in $\bbR^n$ so that an outer unit normal of $\partial \Omega_1$ is $E_1$ and $-tE_1 \in  \operatorname{int}(\Omega_1)$ for all sufficiently small $t>0$. Note that
\eq{
	0&=\hat{s}_{\Omega_1}(E_1)=\frac{\hat{s}_{\widehat{\Sigma}_1}(\sin \theta E_1 +\cos\theta E_{n+1})}{\sin \theta},\\
		0&<\hat{s}_{\Omega_1}(-E_1)=\frac{\hat{s}_{\widehat{\Sigma}_1}(-\sin \theta E_1 +\cos\theta E_{n+1})}{\sin \theta}.
}
Here, $\hat{s}_{\Omega_1}(-E_1)>0$ follows from the earlier observation that $\Omega_1$ has non-empty interior.
By the first identity, $\widehat{\Sigma}_1\subset \{ y\in \overline{\bbR^{n+1}_+}~|\sin\theta y_1+\cos\theta y_{n+1}\leq 0 \}$.
Now we may proceed as in the proof of \autoref{steiner}  with the caveat that $i(\zeta^+)$ is defined as a point (rather than the unique point) in $\widehat{\Sigma}_1$ so that
\eq{
\hat{s}_{\widehat{\Sigma}_1}(\zeta^++\cos\theta E_{n+1}) = \langle \zeta^++\cos\theta E_{n+1}, i(\zeta^+) \rangle.
}
\end{proof}

\begin{rem}
It is not clear to us whether a reflection argument could be used to prove the last statement in \autoref{steiner} for the case $\psi(x)=x^2$. That is, if
\eq{\label{eqxxx}
\int_{\cC_{\theta}} s_{\Sigma}(\zeta)\zeta_i \, d\si = 0, \quad i = 1, \ldots, n,
}
then $s>0$.  Note that from \eqref{eqxxx}, we obtain
\eq{
\int_{\bbS^n} \hat{s}_K(x) x_i \, d\si = \int_{R} \hat{s}_K(x) x_i \, d\si,
}
where $R$ is the complement of $\bbS^n_{\theta} \cup (-\bbS^n_{\theta})$ in $\bbS^n$ and $K$ is the enclosed region by $\Sigma$ and its reflection across $x_{n+1}=0$.

If $\widehat{\Sigma} \cap E_{n+1}^{\perp}$ is origin-symmetric, then the integral over $R$ vanishes. Indeed, in this case, we can simply argue as follows:
\eq{
\hat{s}_K(x) = \langle x, \tilde{x} \rangle \hat{s}_{\Omega}(\tilde{x}), \quad x \in R,
}
where $\tilde{x}$ is the projection of $x$ onto $E_{n+1}^{\perp}$ and normalized to have length one, and $\hat{s}_{\Omega}$ is the support function of the convex set $\Omega = K \cap E_{n+1}^{\perp}$. Given $x = (x_1, \ldots, x_n, x_{n+1}) \in R$, consider the reflected point  
\eq{
y = (-x_1, \ldots, -x_n, x_{n+1}) \in R.
}
Then $\tilde{y} = -\tilde{x}$ and, by the origin symmetry of $\Omega$, we have $\hat{s}_{\Omega}(\tilde{y}) = \hat{s}_{\Omega}(\tilde{x})$. Moreover,
\eq{
\langle x, \tilde{x} \rangle = \langle y, \tilde{y} \rangle, \quad \text{while} \quad \langle y, E_i \rangle = -\langle x, E_i \rangle.
}
Thus, the contributions from $x$ and $y$ cancel in the integral over $R$, and
\eq{
\int_{\bbS^n} \hat{s}_K(x) x_i \, d\si = 0, \quad i = 1, \ldots, n.
}
For $i = n + 1$, due to the reflection symmetry of $K$, we have
\eq{
\int_{\bbS^n} \hat{s}_K(x) x_{n+1} \, d\si = 0.
}
Hence, the Steiner point of $K$ is at the origin and $s_{\Sigma} > 0$, see \cite[(1.34)]{Sch14}.

However, if $\Omega$ is not origin-symmetric, we were not able to conclude that the integral over $R$ vanishes.
\end{rem}

\begin{lemma}\label{upper support func bound}
Let $\theta \in (0, \pi/2)$ and $1 \leq k \leq n$. Suppose $\Sigma$ is a strictly convex capillary hypersurface and $s_{\Sigma}$ satisfies 
\begin{align}\label{entropy l2} 
\int_{\cC_{\theta}} s_{\Sigma}(\zeta) \zeta_i \, d\si = 0, \quad \forall i = 1, \ldots, n.
\end{align}
and $\sigma_k(\tau^{\sharp}[s_{\Sigma}]) = \phi$. Then $0 < s_{\Sigma} \leq C$ for some constant depending on $\theta$ and $\phi$. Moreover, for $k = n$, we have
$
c \leq s_{\Sigma - w} \leq C'
$
for some vector $w \in \widehat{\Sigma} \cap E_{n+1}^{\bot}$ and for some constants $c, C'$ depending on $\theta$ and $\phi$.
\end{lemma}
\begin{proof} In view of \autoref{steiner} with $\psi(x)=x^2$, $s>0$, hence, $\hat{s}\geq 0$. By integrating by parts (cf. \cite[Cor. 2.10]{MWWX24}) and using the Newton-Maclaurin's inequality, we find 
\eq{
c'_{n,k}\int_{\cC_{\theta}} s \phi^{\frac{k-1}{k}}\, d\si\leq \int_{\cC_\theta} s\sigma_{k-1} \, d\si= c_{n,k}\int_{\cC_\theta} \ell \sigma_k \, d\si=c_{n,k}\int_{\cC_\theta} \ell \phi \, d\si.
}
Therefore,
\eq{
\int_{\bbS^n_{\theta}} \hat{s} \, d\si=\int_{\cC_{\theta}} s \, d\si\leq C.
}

Suppose the maximum of $\hat{s}$ is attained for some $u_0 \in \bbS^n \setminus \operatorname{int}(\bbS^n_{\theta})$. Note that $u_0\neq -E_{n+1}$, and in the $\bbS^n \setminus \operatorname{int}(\bbS^n_{\theta})$ there holds $\hat{s} = \hat{s}_{\Omega}$, where $\Omega = \widehat{\Sigma} \cap E_{n+1}^{\bot}$ is considered as a convex set in $\mathbb{R}^{n+1}$. Write $u_0 = \sin \tilde{\theta} \tilde{u}_0 \pm \cos \tilde{\theta} E_{n+1}$ for some $\tilde{\theta} \in (0, \pi/2]$ and $\tilde{u}_0 \in \bbS^{n-1}$. Therefore, $\hat{s}(u_0) = \sin \tilde{\theta} \hat{s}_{\Omega}(\tilde{u}_0)$. This implies that the maximum of $\hat{s}$ must be attained for some $u_0 = \tilde{u}_0 \in E_{n+1}^{\bot}$ (i.e., $\tilde{\theta} = \pi/2$). In this case, note that $\max_{\bbS^{n-1}} \hat{s}_{\Omega}$ is also attained at $u_0$, hence, we have
\eq{
(\max \hat{s})u_0 = \hat{s}_{\Omega}(u_0)u_0\in \partial \Sigma\subset \widehat{\Sigma}.
}  

Let us assume $u_0=E_1$. By the definition of $\hat{s}$, for $x\in \bbS_{\theta}^n$: 
\eq{
\hat{s}(x)\geq x_1\max \hat{s}
}
and hence,  
\eq{
\int_{\bbS^n_{\theta}} \hat{s}d\si\geq \int_{\{x_1\geq \frac{1}{2}\sin \theta\}\cap \bbS^n_{\theta}} \hat{s} \, d\si\geq (\max \hat{s})\frac{\sin \theta}{2}\int_{\{x_1\geq \frac{1}{2}\sin \theta\}\cap \bbS^n_{\theta}} \, d\si.
}

Suppose the maximum of $\hat{s}$ is attained at $u_0\in \operatorname{int}(\bbS^n_{\theta})$. Therefore, we have  $(\max \hat{s})u_0\in \Sigma$. We may assume $u_0=(\sin \theta_0,0,\ldots,0,\cos \theta_0)$ with $0\leq\theta_0<\theta$. Therefore, 
\eq{
\int_{\bbS^n_{\theta}} \hat{s} \, d\si\geq \int_{\{x_1\geq 0\}\cap \bbS^n_{\theta}} \hat{s} \, d\si
&\geq \int_{\{x_1\geq 0\}\cap \bbS^n_{\theta}} (\max \hat{s}) x_{n+1}\cos\theta_0 \, d\si\\
&\geq (\max \hat{s})\cos\theta \cos\theta_0  \int_{\{x_1\geq 0\}\cap \bbS^n_{\theta}}  d\si\\
&\geq (\max \hat{s})\cos^2\theta  \int_{\{x_1\geq 0\}\cap \bbS^n_{\theta}}  d\si.
}
Next, we prove the lower bound on $s$ when $k=n$. Note that
\eq{
\int_{\cC_{\theta}}s\sigma_n  \, d\si&=\int_{\cC_{\theta}} s\phi  \, d\si\geq \min \phi \int_{\cC_{\theta}} s \, d\si\geq c_{n,\theta}\min \phi \left(\int_{\cC_{\theta}}s\sigma_n \, d\si\right)^{\frac{1}{n+1}}.
}
To obtain the right-hand side inequality, we used \cite[Thm 1.1]{MWWX24}, from which it follows that the sequence  
\eq{
a_i = V(\widehat{\cC}_{\theta},\dots, \widehat{\cC}_{\theta},\widehat{\Sigma},\dots,\widehat{\Sigma}),\q i=0,\dots,n+1,
}  
where $\widehat{\Sigma}$ appears $i$ times, satisfies
\begin{equation}
\frac{a_1}{a_0} \geq \frac{a_2}{a_1} \geq \cdots \geq \frac{a_{n+1}}{a_n}.
\end{equation}
In particular, we have
\eq{
\frac{a_1}{a_0} \geq \left(\frac{a_{n+1}}{a_0}\right)^{\frac{1}{n+1}}.
}
Thus, $V(\widehat{\Sigma})$ is bounded below by a constant that depends only on $\theta$ and $\phi$. 

Define 
\eq{
w_- = \min_{x\in \bbS^n} (\hat{s}(-x)+\hat{s}(x)), \quad w_+ = \max_{\bbS^n} (\hat{s}(-x)+\hat{s}(x)).
}
Then we have $V(\widehat{\Sigma}) \leq w_- w_+^n$. Moreover, by \cite[p. 320]{Sch14}, we have
\begin{equation}
\frac{w_-}{n+2} B \subseteq \widehat{\Sigma} - \operatorname{centroid}(\widehat{\Sigma}),
\end{equation}
where $B$ denotes the unit ball.  Therefore, the radius of the inner ball of $\widehat{\Sigma}$ satisfies $r \geq \frac{w_-}{n+2}$. Since $\hat{s}$ is bounded above (cf. \autoref{cap s bound to standard s bound}), we can find $r = r(\theta,\phi)$ and a point $x \in \operatorname{int}(\widehat{\Sigma})$ such that  
\eq{
B_{r}(x) \subseteq \widehat{\Sigma}.
}
Now, from \cite[(2.26)]{SW24} it follows that 
$
\widehat{\cC}_{\theta,\frac{r}{\sin \theta}}\subseteq \widehat{\Sigma}-w,
$
where  $w\in \widehat{\Sigma}\cap E_{n+1}^{\bot}$ is the projection of $x$ onto $E_{n+1}^{\bot}$. By \autoref{cap s bound to standard s bound},
$
|w|\leq \frac{C}{\sin \theta}.
$
Hence,  
\eq{s_{\cC_{\theta,\frac{r}{\sin \theta}}}\leq s_{\Sigma-w}\leq C\left(1+\frac{1}{\sin \theta}\right).
}
\end{proof}
\begin{rem}\label{rem: max s}
In the proof of \autoref{upper support func bound}, we also proved that when $\theta\in (0,\pi/2)$, $\hat{s}$ attains its maximum either in $\operatorname{int}(\bbS^n_{\theta})$ or in $\bbS^n \cap E_{n+1}^{\bot}$. In particular, the maximum of $\hat{s}$ is never attained on $\partial \bbS^n_{\theta}$.
\end{rem}

\begin{lemma}\label{upper grad support func bound}
Let $\theta\in (0,\pi)$ and $\Sigma$ be a strictly convex capillary hypersurface  with $0<s_{\Sigma}\leq C$. Then $|\bar{\nabla} s|\leq \frac{C}{\sin\theta}$.
\end{lemma}
\begin{proof}
Set $\rho^2=s^2+|\bar{\nabla} s|^2$. Suppose $\rho$ attains its maximum at $p$. If  $p\in \operatorname{int}(\cC_{\theta})$, then $\bar{\nabla} s|_p=0$ and $\rho^2\leq C^2$. Suppose $p\in\partial \cC_{\theta}$. Let $\{\mu\} \cup \{e_{\alpha}\}_{\alpha\geq 2}$ be an orthonormal basis of eigenvectors of $\tau^{\sharp}$ at $p$.  Using \autoref{key-lemma-capillary-function} we find 
\eq{
0=\bar{\nabla}_{e_\al}\rho^2=2s\bar{\nabla}_{e_\al}s+2\bar\nabla_{e_\beta}s\bar\nabla^2s(e_{\beta},e_{\al})=2\la_{\al}\bar{\nabla}_{e_\al}s.
}
Therefore,
$\bar\nabla_{e_\al}s|_p=0$ for all $\al=2,\ldots, n$ and $|\bar\nabla s|=|\bar \nabla_{\mu}s|=\abs{\cot \theta} s$. Hence, 
\eq{
\rho^2(p)\leq \frac{C^2}{\sin^2\theta}.
}

\end{proof}

\begin{lemma}\label{sigma 1 bound}
Let $\theta\in (0,\pi/2)$ and $1\leq k\leq n$.
Suppose $\Sigma$ is a strictly convex capillary hypersurface and $s_{\Sigma}$ satisfies $\sigma_k(\tau^{\sharp}[s_{\Sigma}])=\phi$. Then,
$\sigma_1(\tau^{\sharp}[s_{\Sigma}])\leq C
$
for some constant depending only on $\theta$ and $\phi$.
\end{lemma}
\begin{proof}
Recall that $F = \sigma_k^{\fr 1k}$ and $f = \phi^{\fr 1k}$. If the maximum of $\sigma_1$ is attained in the interior of $\cC_{\theta}$, then
\eq{
0 \geq F^{ij}\bar\nabla^2_{i,j}\si_1&= -\bar g^{mn}F^{ij,pq}\bar\nabla_{m}\tau_{ij}\bar\nabla_n \tau_{pq} -nf+\bar\Delta f+F^{ij}\bar g_{ij}\si_1 \\
&\geq -nf+\bar\Delta f+F^{ij}\bar g_{ij}\si_1,
}
where we used $\bar{g}^{mn} F^{ij,pq} \bar{\nabla}_{m} \tau_{ij} \bar{\nabla}_{n} \tau_{pq} \leq 0$ (due to the concavity of $F$) and the one-homogeneity of $F$. The claim follows from $\operatorname{tr}(\dot{F}) \geq c_{n,k}$ (again due to the concavity of $F$).

Now we need to treat the case that the maximum of $\sigma_1$ is attained at a boundary point, say $p$. Then at $p$ we have
\eq{\label{second proof}
0\leq F^{\mu\mu}\bar\nabla_{\mu}\sigma_1&=\cot\theta \sum_{i}(F^{\mu\mu}-F^{ii})(\la_{\mu}-\la_{i})+\bar{\nabla}_{\mu}f\\
&= \cot \theta (f-F^{\mu\mu}\sigma_1+\sum_iF^{\mu\mu}\la_{\mu}-\sum _{i}F^{ii}\la_{\mu})+\bar{\nabla}_{\mu}f\\
&\leq \cot \theta \left((n+1)f-F^{\mu\mu}\sigma_1-\sum _{i}F^{ii}\la_{\mu}\right)+\bar{\nabla}_{\mu}f.
}
Therefore, 
\eq{\label{si1 bound 2nd proof}
\sigma_1(p)\leq \frac{\max |\bar{\nabla}_{\mu}f|}{\cot \theta F^{\mu\mu}(p)}+(n+1)\frac{f(p)}{F^{\mu\mu}(p)}.
}
Next we show that $F^{\mu\mu}$ cannot be very small. Note that
\eq{\label{dot F mumu estimate}
c:=\min \phi\leq \phi=\si_k(\la)&=\la_\mu \si_{k-1}(\la|\la_\mu)+\si_k(\la|\la_\mu)\\
                     &\leq \la_\mu  \si_{k-1}(\la|\la_\mu)+c_0\si_{k-1}(\la|\la_\mu)^\frac{k}{k-1}\\
                     &\leq c_2 \la_{\mu}F^{\mu\mu}+c_3 (F^{\mu\mu} )^{\frac{k}{k-1}}.
}
Using this in \eqref{second proof}, we obtain
\eq{
0\leq F^{\mu\mu}\bar\nabla_{\mu}\sigma_1&\leq ((n+1)f-\sum _{i}F^{ii}\la_{\mu})\cot \theta +\bar{\nabla}_{\mu}f\\
&\leq (n+1) f\cot \theta+ \frac{1}{c_2}\sum_{i}F^{ii}\left(-\frac{c}{F^{\mu\mu}}+c_3(F^{\mu\mu})^{\frac{1}{k-1}}\right)+\bar{\nabla}_{\mu}f.
}
Due to concavity of $F$, we have $\operatorname{tr}(\dot{F})\geq c_{n,k}$. Hence $F^{\mu\mu}$ cannot be small, and in view of \eqref{si1 bound 2nd proof}, $\sigma_1$ is bounded above and the bound depends only on $\theta,\phi$.

\end{proof}

\begin{prop}\label{prop: reg estim 2}
Let $\theta\in (0,\pi/2)$ and $1\leq k\leq n$. Suppose $\Sigma$ is a strictly convex capillary hypersurface whose $s_{\Sigma}$ satisfies \eqref{entropy l2} and $\sigma_k(\tau^{\sharp}[s_{\Sigma}])=\phi$. Then for any $m\geq 1$ we have $\|s_{\Sigma}\|_{C^{m}}\leq C_m$ for some constant depending only on $\theta, \phi$. 
\end{prop}

\begin{proof}
By \autoref{steiner} and \autoref{upper support func bound}, we have $0 < s \leq C_1$. Hence, by \autoref{upper grad support func bound} and \autoref{sigma 1 bound}, $\|s\|_{C^2}\leq C_2$ for some constant depending only on $\theta, \phi$. The higher-order regularity follows from \cite{LT86} and Schauder estimate.

\end{proof}

\section{Continuity method}

To prove \autoref{CM-thm1}, we need to construct a suitable homotopy path that stays within the appropriate function class and along which we can establish uniform estimates. We now present the necessary preliminary results.

\begin{lemma}\label{non-collapsing En+1 dirc}Let $\theta \in (0, \pi/2]$. Consider a sequence of strictly convex capillary hypersurfaces $\{\Sigma_i\}$ with uniform $C^2$-norm bound on their capillary support functions $\{s_{\Sigma_i}\}$ and satisfying \eqref{entropy l2}. Suppose $\sigma_k(\tau^{\sharp}[s_{\Sigma_i}]) \geq c_0 > 0$ for all $i$. Then there is a constant $c > 0$ such that
\eq{
\hat{s}_{\widehat{\Sigma}_i}(E_{n+1}) \geq c \quad \forall i.
}
\end{lemma}
\begin{proof}
Suppose that along a subsequence, $\hat{s}_{\widehat{\Sigma}_i}(E_{n+1})\to 0$. Hence, due to \autoref{cap s bound to standard s bound}, $\widehat{\Sigma}_i$ converges in the Hausdorff distance to a convex set $K\subset E_{n+1}^{\bot}$. Recall that $s_{\Sigma_i}(\zeta)=\hat{s}_{\widehat{\Sigma}_i}(\zeta +\cos \theta E_{n+1})$ for all $\zeta \in \cC_{\theta}$ and by our assumption $\|s_{\Sigma_i}\|_{C^2}$ is uniformly bounded above. Therefore, $\hat{s}_{K}$ is differentiable at $E_{n+1}$, and by \cite[Cor. 1.7.3]{Sch14}, the set
\eq{
F(K,E_{n+1}):=E_{n+1}^{\bot}\cap K=K
}
must contain only one point. By \autoref{steiner}, $0\in \operatorname{int}(\widehat{\Sigma}_i\cap E_{n+1}^{\bot})$. Therefore, $K=\{0\}$, $\hat{s}_{\widehat{\Sigma}_i}\to 0$ and $s_{\Sigma_i}\to 0$ uniformly (i.e. in the $\sup$-norm). 

Now, by integrating by parts (cf. \cite[Cor. 2.10]{MWWX24}) and the uniform upper bound on $\sigma_{k-1}^{\Sigma_i}$, we find 
\eq{
C\int_{\cC_\theta} s_{\Sigma_i} \, d\si&\geq
\int_{\cC_\theta} s_{\Sigma_i}\sigma_{k-1}^{\Sigma_i} \, d\si= c_{n,k}\int_{\cC_\theta} \ell \sigma_k^{\Sigma_i} \, d\si\geq c_0c_{n,k}\int_{\cC_\theta} \ell  \, d\si.
}
Taking the limit as $i\to\infty$, we arrive at a contradiction.
\end{proof}

\begin{thm}\label{lower support bound steiner}
Under the assumptions of \autoref{non-collapsing En+1 dirc}, 
let $\Omega_i=\widehat{\Sigma}_i\cap E_{n+1}^{\bot}$ and assume that $\max_{\bbS^{n-1}} \hat{s}_{\Omega_i}\geq c_1>0$ for all $i$, where $\hat{s}_{\Omega_i}$ denotes the support function of $\Omega_i$ as a convex body in $\bbR^n$.
Then we have $V(\widehat{\Sigma}_i)\geq c_2>0$ for all $i$.
\end{thm}
\begin{proof}
Since $\max_{\bbS^{n-1}} \hat{s}_{\Omega_i} \geq c_1 $, $0 \in \operatorname{int}(\Omega_i)$, and 
\eq{
\hat{s}_{\Omega_i}(x) + \hat{s}_{\Omega_i}(-x) \leq \operatorname{diam}(\Omega_i)\leq 2R_{\Omega_i}\q \forall x\in \bbS^{n-1},
} 
we have $R_{\Omega_i} \geq c_1/2$. Therefore, due to \autoref{app: CW}, $r_{\Omega_i}$ is uniformly bounded below away from zero. 

Recall from \autoref{non-collapsing En+1 dirc} that $\hat{s}_{\widehat{\Sigma}_i}(E_{n+1})$ is uniformly bounded below away from zero. Considering the convex hull of the inner ball of $\Omega_i$ with radius $r_{\Omega_i}$ and the point on $\Sigma_i$ whose supporting hyperplane is parallel to $E_{n+1}^{\bot}$ (which forms an oblique cone that can be transformed into a right cone of the same height via a volume-preserving shear transformation that acts as the identity on $x_{n+1} = 0$), we conclude that $V(\widehat{\Sigma}_i) \geq c$ for some $c>0$ and all $i$.  
\end{proof}

\begin{thm}\label{homotopy path}
Let $\theta \in (0,\frac{\pi}{2})$  and $1\leq k< n$. Suppose $0< \phi\in C^{\infty}(\cC_{\theta})$ is a function satisfying
\eq{
\tau[ \phi^{-\frac{1}{k}}]\geq 0,\q  \bar{\nabla}_{\mu}\phi^{-\frac{1}{k}}=\phi^{-\frac{1}{k}}\cot\theta, \q \int_{\cC_{\theta}}  \phi(\zeta)\zeta^i \, d\si=0,\q i=1,\ldots, n.
}
Then for each $t\in [0,1]$, there exists $z_t\in E_{n+1}^{\bot}$, such that
\eq{
 \phi_t(\zeta):=((1-t)\ell(\zeta)+t \phi^{-\frac{1}{k}}(\zeta)-\langle \zeta,z_t\rangle)^{-k}
}
satisfies
\eq{
 \int_{\cC_{\theta}}  \phi_t(\zeta)\zeta^i \, d\si=0,\q i=1,\ldots, n\\
 \tau[ \phi_t^{-\frac{1}{k}}]\geq 0,\q \bar\nabla_{\mu}\phi_t^{-\frac{1}{k}}=\cot \theta \phi_t^{-\frac{1}{k}}.
}
Moreover, $z_0=z_1=0$ and for some constant $C$ we have
\eq{
\frac{1}{C}\leq \phi_t\leq C, \q \forall t\in [0,1].
}
\end{thm}
\begin{proof}
For $0\leq t<1$, by \autoref{constr hyper from s}, $s_t:= (1-t)\ell+t \phi^{-\frac{1}{k}}$ is the capillary support function of a strictly convex capillary hypersurface $\Sigma_t$. By \autoref{steiner}, there is a unique point $z_t$ in the interior of $\Omega_t=\widehat{\Sigma}_t\cap E_{n+1}^{\bot}$ so that
\eq{
 \left\{
\begin{array}{ll}
\inf_{v\in \operatorname{int}(\Omega_t)}\int_{\cC_{\theta}}-\log (s_{\Sigma_t-v}) \, d\sigma & k=1, \\
\inf_{v\in \operatorname{int}(\Omega_t)}\int_{\cC_{\theta}} \frac{1}{s^{k-1}_{\Sigma_t-v}} \, d\si & \, k>1 .
\end{array}
\right.
}
is attained. Thus, for $t\in [0,1)$, $s_{\Sigma_t-z_t}= \phi_t^{-\frac{1}{k}}>0$ satisfies the required integral condition and
\eq{\label{req intergral}
 \tau[ \phi_t^{-\frac{1}{k}}]> 0,\q
\bar\nabla_\mu \phi_{t}^{-\frac{1}{k}}=\phi_t^{-\frac{1}{k}}\cot \theta.
}
For $t = 1$, $k \geq 1$, we have the same conclusion by \autoref{steiner2}. Next we show that $z_1=0$. We consider the case $k > 1$. As in the proof of \autoref{steiner2}, the sequence of convex bodies $\{\widehat{\Sigma}_t\}$ converges in the Hausdorff distance to a convex body $\widehat{\Sigma}_1$ with $\hat{s}_{\widehat{\Sigma}_1}(\cdot) = \phi^{-\frac{1}{k}}(\cdot - \cos\theta\, E_{n+1})$.
 Let $\Omega_1=\widehat{\Sigma}_1 \cap E_{n+1}^{\bot}$.
By \eqref{key observations} with $h = \phi^{-\frac{1}{k}}$, we have $\hat{s}_{\Omega_1} > 0$. Therefore, $0 \in \operatorname{int} (\Omega_1)$.  
Since 
\eq{
\int_{\cC_{\theta}} \frac{\zeta_i}{(\phi^{-\frac{1}{k}}(\zeta))^k} \, d\si = 0,
}
and due to uniqueness of $z_1$, we conclude that $z_1 = 0$.

Next we prove $Q: [0,1]\times \cC_{\theta}\to \bbR$ defined as $Q(t,\zeta)=s_{\Sigma_t-z_t}(\zeta)$ is continuous. Here, $\Sigma_1:= \partial\widehat{\Sigma}_1\setminus \operatorname{int}(\Omega_1)$ and $s_{\Sigma_1}:=\phi^{-\frac{1}{k}}$. Suppose $t_i\in [0,1)$ and $t_i\to a$ and $z_i:=z_{t_i}$ converge to $z\neq z_a$ (after passing to a subsequence). Then for $t_i$ sufficiently close to $a$ we have
\eq{
\int_{\cC_{\theta}} \frac{1}{(s_{\Sigma_{t_i}-z_i})^{k-1}} \, d\si<\int_{\cC_{\theta}} \frac{1}{(s_{\Sigma_{t_i}-z_a})^{k-1}} \, d\si,
}
where we used $z_a\in \operatorname{int}(\Omega_{t_i})$. Moreover, there holds
\eq{
\limsup_{t_i\to a}\int_{\cC_{\theta}} \frac{1}{(s_{\Sigma_{t_i}-z_i})^{k-1}} \, d\si&\leq \limsup_{t_i\to a}\int_{\cC_{\theta}} \frac{1}{(s_{\Sigma_{t_i}-z_a})^{k-1}} \, d\si=\int_{\cC_{\theta}} \frac{1}{(s_{\Sigma_{a}-z_a})^{k-1}} \, d\si.
}
On the other hand, by Fatou's lemma,
\eq{
\limsup_{t_i\to a}\int_{\cC_{\theta}} \frac{1}{(s_{\Sigma_{t_i}-z_i})^{k-1}} \, d\si&\geq \liminf_{t_i\to a}\int_{\cC_{\theta}} \frac{1}{(s_{\Sigma_{t_i}-z_i})^{k-1}} \, d\si\geq \int_{\cC_{\theta}} \frac{1}{(s_{\Sigma_{a}-z})^{k-1}} \, d\si.
}
Therefore,
\eq{
\int_{\cC_{\theta}} \frac{1}{(s_{\Sigma_{a}-z_a})^{k-1}} \, d\si\geq \int_{\cC_{\theta}} \frac{1}{(s_{\Sigma_{a}-z})^{k-1}} \, d\si.
}
which violates the uniqueness of $z_a$. The argument for $k=1$ is similar.
 
Since $Q$ is continuous and positive, it is bounded below  away from zero. Moreover, by \autoref{cap s bound to standard s bound}, for $0< t<1$, we have
\eq{
|z_t|\leq \frac{\sin^2 \theta+\max  \phi^{-\frac{1}{k}}}{\sin \theta}.
}
Therefore, $\phi_t^{-\frac{1}{k}}$ is bounded below.
\end{proof}

Let $S_{n}$ be the group of all permutations of $\{1,\ldots,n\}$ and 
\eq{
\operatorname{sgn}:S_{n}\to \{-1,1\}
}
 be defined by $\operatorname{sgn}(\al)=1$ $(-1)$ if $\al$ is even (odd).

For smooth functions $f_i\in C^{2}(\cC_{\theta})$ with $\bar \nabla_{\mu}f_i=\cot \theta f_i$, we define
\eq{
Q[f_1,\ldots,f_n]=\frac{1}{n!}\sum_{\al,\,\beta \in S_n}(-1)^{\operatorname{sgn}(\al)+\operatorname{sgn}(\beta)}\tau^{\sharp}[f_1]_{\al(1)}^{\beta(1)}\cdots \tau^{\sharp}[f_n]_{\al(n)}^{\beta(n)}.
}

Let $s$ be the capillary support function of a strictly convex capillary hypersurface. Note that 
\eq{
\sigma_k(\tau^{\sharp}[s])=c_{k}Q[s,\ldots,s,\ell,\ldots,\ell],
}
where $s$ appears $k$ times. Define the linearized operator 
\eq{
L_{s,k}[\upsilon]=\frac{d}{dt}_{|_{t=0}}\sigma_k(\tau^{\sharp}[s+t\upsilon]),
}
for all capillary functions  $\upsilon\in C^2(\cC_{\theta})$.

\begin{lemma}\label{trivial kernel CM case}
Suppose $L_{s,k}[\upsilon]=0$ where $\upsilon\in C^2(\cC_{\theta})$ is a capillary function. Then $\upsilon(\zeta)=\langle \vec{a},\zeta\rangle$ for some vector $\vec{a}\in E_{n+1}^{\bot}$.
\end{lemma}
\begin{proof}
Let us define $Q_k[\upsilon] = Q[\upsilon, s, \ldots, s, \ell, \ldots, \ell]$, where $\ell$ appears $n-k$ times. Then, for any $\upsilon \in C^2(\cC_{\theta})$ with $\bar \nabla_{\mu} \upsilon = \cot \theta \upsilon$, by the (local) Aleksandrov-Fenchel inequality \cite[Thm. 3.1]{MWWX24}, we have
\eq{
V^2(\upsilon,s,\ldots,s,\ell,\ldots,\ell)^2\geq V(\upsilon,\upsilon,s,\ldots,s, \ell,\ldots,\ell)V(s,\ldots,s,\ell,\ldots,\ell),
}
where $\ell$ appears $n-k$ times on the left-hand and right-hand sides. That is,
\eq{\label{local AF}
\left(\int_{\cC_{\theta}} \upsilon Q_k[s]  \, d\si\right)^2 \geq \int_{\cC_{\theta}} \upsilon Q_k[\upsilon] \, d\si \int_{\cC_{\theta}} s Q_k[s] \, d\si.
}
Moreover, equality holds if and only if 
\eq{
\upsilon(\zeta) = a s(\zeta) + \langle \vec{a}, \zeta \rangle
} 
for some vector $\vec{a} \in E_{n+1}^{\bot}$ and constant $a\geq 0$.

Now note that
\eq{
\sigma_k(\tau^{\sharp}[s + t \upsilon]) = c_k Q[s + t \upsilon, \ldots, s + t \upsilon, \ell, \ldots, \ell],
}
where $s + t \upsilon$ appears $k$ times. Hence,
\eq{
0 = L_{s,k}[\upsilon] = k c_{k} Q_k[\upsilon].
}
Since $\int_{\cC_{\theta}} \upsilon Q_k[s] \, d\si = \int_{\cC_{\theta}} s Q_k[\upsilon]\,  d\si= 0$, we have equality in \eqref{local AF}, and the claim follows (from $Q_k[\upsilon] = 0$, we deduce that $a = 0$).
\end{proof}

In virtue of \autoref{prop: reg estim 2} and \autoref{trivial kernel CM case}, we can proceed as follows to finish the existence of solutions as claimed in \autoref{Min P},  \autoref{CM-thm} and \autoref{CM-thm1}. We start with the first two theorems. Let $t\in [0,1]$. Define $\phi_t=1-t+t\phi$ for $k=n$ and $\phi_t=(1-t+t\phi^{-\frac{1}{k}})^{-k}$ for $1\leq  k<n$.  Recall that when $1\leq k<n$ additionally we assumed 
\eq{
\phi(-\zeta_1,\ldots, -\zeta_n,\zeta_{n+1})=\phi(\zeta_1,\ldots,\zeta_n, \zeta_{n+1}) ,\q \forall \zeta\in \cC_{\theta}.
}
Therefore, we have
\eq{
\int_{\cC_\theta} \langle \zeta,E_{i}\rangle \phi_t(\zeta) \, d\si =0, \quad i=1,2,\ldots,n, 
}
and when $1\leq k<n$, we also have
\eq{
\bar \nabla^2 \phi_t^{-\frac{1}{k}}+\bar g \phi_t^{-\frac{1}{k}}\geq 0,\q\bar\nabla_{\mu} \phi_t^{-\frac{1}{k}}\leq \cot\theta \phi_t^{-\frac{1}{k}}.
}

Define, for fixed $m$,
\eq{
\cA_k = \left\{ t \in [0,1]: \exists s_t \in C^{m+2,\alpha}(\mathcal{C}_{\theta})~ \text{with}~  
\begin{aligned}
    &\sigma_k(\tau[s_t],\bar{g}) = \phi_t, \\
    &\bar{\nabla}_{\mu} s_t = \cot \theta  s_t,\,\tau[s_t] > 0, \\
    &\int_{\mathcal{C}_{\theta}} s_t(\zeta) \zeta_i \, d\sigma = 0,\,\forall i = 1, \dots, n
\end{aligned} \right\}.
}
\emph{Case $k=n$}: Due to \autoref{prop: reg estim 2}, the set $\cA_n$ is closed. 
Regarding the openness, define the Banach spaces 
\eq{
B_1=\{s\in C^{m+2,\alpha}(\cC_{\theta}):\, \bar \nabla_{\mu}s=\cot \theta s,\q \int_{\cC_{\theta}} s(\zeta)\zeta_id\si =0,\q \forall i=1,\ldots n \}
}
and 
\eq{
B_2=\{\phi\in C^{m,\alpha}(\cC_{\theta}):\,\int_{\cC_{\theta}} \phi(\zeta)\zeta_id\si =0,\q \forall i=1,\ldots n \}.
}
Due to \autoref{trivial kernel CM case}, the self-adjointness of $L_{s,n}$ and by the implicit function theorem, the transformation
\begin{align*}
\Phi: B_1\to B_2,\q \Phi(s)=\sigma_n(\tau[s],\bar g)
\end{align*}
is locally invertible at any $s$ with $\tau[s]>0$, provided $m$ is sufficiently large. Hence, $\cA_n$ is open. Since $0\in \cA_n$, we have $\cA_n=[0,1]$ and there exists a strictly convex solution to $\sigma_n(\tau^{\sharp}[s]) = \phi$.

\emph{Case $1\leq k<n$}: 
 Let us write $C_{e}(\cC_{\theta})$ for the set of functions $f$ in $C(\cC_{\theta})$ with 
\eq{
f(-\zeta_1,\ldots, -\zeta_n,\zeta_{n+1})=f(\zeta_1,\ldots,\zeta_n, \zeta_{n+1}),\q \forall \zeta \in \cC_{\theta}
} 
 and $C_e^{m,\alpha}(\cC_{\theta})=C^{m,\alpha}(\cC_{\theta})\cap C_e(\cC_{\theta})$. Now define
 
\eq{
S = \left\{ t \in [0,1]: \exists s_t \in C_e^{m+2,\alpha}(\mathcal{C}_{\theta})~ \text{with}~  
\begin{aligned}
    &\sigma_k(\tau[s_t],\bar{g}) = \phi_t, \\
    &  \quad \tau[s_t] > 0,\\
    &\bar{\nabla}_{\mu} s_t = \cot \theta \, s_t
\end{aligned} \right\}
}

and
\eq{
B_1&=\{s\in C_e^{m+2,\alpha}(\cC_{\theta}):\, \bar \nabla_{\mu}s=\cot \theta s\},\q
B_2= C_e^{m,\alpha}(\cC_{\theta}),\\
\Ph&: B_1\to B_2,\q \Phi(s)=\sigma_k(\tau[s],\bar g).
}
Then $\phi$ is locally invertible at any $s$ with $\tau[s]>0$ and thus $S$ is open.

To show that $S$ is closed, we need the full rank theorem. Suppose $t_i  \in S$ with $t_i \to t_0$. Let $\Sigma_{i}$ denote the strictly convex capillary hypersurface with the capillary support function $s_{t_i}$ (see \autoref{constr hyper from s}) and put $\Omega_{i}=\widehat{\Sigma}_{i}\cap E_{n+1}^{\bot}$. By \autoref{prop: reg estim 2}, after passing to a subsequence, $\{\widehat{\Sigma}_i\}$ converges in the Hausdorff distance to a convex set $\widehat{\Sigma}$, and $s_{t_i}$ converges to a solution $s\geq 0$ of $\sigma_k(\tau^{\sharp}[s]) = \phi_{t_0}$ with $\tau[s] \geq 0$. 

Note that for all $i$ we have 
\eq{
s_{\Sigma_i}(\zeta)=\hat{s}_{\widehat{\Sigma}_i}(\zeta+\cos\theta E_{n+1}),\q \forall \zeta\in \cC_{\theta}
}
and after taking the limit, we find
\eq{
s(\zeta)=\hat{s}_{\widehat{\Sigma}}(\zeta+\cos\theta E_{n+1}),\q \forall \zeta\in \cC_{\theta}. 
}

By \autoref{non-collapsing En+1 dirc},  there exists $w\in  \partial\widehat{\Sigma}$ with $\langle w,E_{n+1}\rangle > 0$. Hence, 
\eq{
\hat{s}_{\widehat{\Sigma}}(x)\geq \langle x,w\rangle, \q \forall x\in \bbS^n.
}
If $w = E_{n+1}$, then for any $x_0 \in \partial \bbS^n_{\theta}$, we have $\hat{s}_{\widehat{\Sigma}}(x_0) > 0$. If $w \neq E_{n+1}$, then for $x_0 = \sin \theta\, \tilde{w} + \cos \theta\, E_{n+1} \in \partial \bbS^n_{\theta}$, we also have $\hat{s}_{\widehat{\Sigma}}(x_0) > 0$. Hence, we conclude that 
$
\max s|_{\partial \cC_{\theta}} > 0.
$
This implies that $\max_{\bbS^{n-1}} \hat{s}_{\Omega_i} \geq \delta$ for some $\delta > 0$ and all $i$. Now, from \autoref{lower support bound steiner}, it follows that $V(\widehat{\Sigma}) \geq c$. Since $s \in C_e(\cC_{\theta})$, we conclude that $s > 0$, and from \autoref{thm: frt}, it follows that $\tau[s] > 0$. Therefore, $t_0 \in S$.

Next, we prove the uniqueness claims. Let $s_1, s_2 \in C^3(\cC_{\theta})$ be two capillary functions with $\tau[s_i] > 0$. In view of \cite[Thm. 1.1]{MWWX24}, the sequence
\eq{
a_i=V(\widehat{\Sigma}_2[k+1-i],\widehat{\Sigma}_1[i],\widehat{\cC}_{\theta},\ldots,\widehat{\cC}_{\theta}), \q i=0,\ldots, k+1,
}
is log-concave i.e.,
\eq{
\frac{a_1}{a_0}\geq \frac{a_2}{a_1}\geq\cdots \geq \frac{a_{k+1}}{a_k}.
}
Here, $\widehat{\Sigma}_i[\cdot]$ indicates the number of times $\widehat{\Sigma}_i$ appears. Moreover, equality holds in either of these inequalities if and only if  
\eq{
\widehat{\Sigma}_2 = a \widehat{\Sigma}_1+ \vec{a}
}  
for some $\vec{a}\in E_{n+1}^{\bot}$ and $a>0$. Therefore, we have
\eq{\label{uniqueness}
\left(\frac{a_k}{a_0}\right)^{\frac{1}{k}}\geq \left(\frac{a_{k+1}}{a_0}\right)^{\frac{1}{k+1}}.
}
That is,
\eq{
\int_{\cC_{\theta}} s_2 \sigma_k(\tau^{\sharp}[s_1]) \, d\si \geq \left(\int_{\cC_{\theta}} s_1 \sigma_k(\tau^{\sharp}[s_1]) \, d\si\right)^{\frac{k}{k+1}} \left(\int_{\cC_{\theta}} s_2 \sigma_k(\tau^{\sharp}[s_2]) \, d\si\right)^{\frac{1}{k+1}}.
}  
Consequently, if $\sigma_k(\tau^{\sharp}[s_1]) = \sigma_k(\tau^{\sharp}[s_2])$, equality in \eqref{uniqueness} follows and thus 
\eq{
s_2(\zeta) = s_1(\zeta) + \langle \vec{a}, \zeta \rangle.
}  
Furthermore, concerning the uniqueness claim in \autoref{CM-thm}, if both of these solutions are in $C_e(\cC_{\theta})$, then we must have $\vec{a} = 0$.

Now, we prove \autoref{CM-thm1}. We define $\phi_t$ as in \autoref{homotopy path}:
\eq{
 \phi_t(\zeta):=((1-t)\ell+t \phi^{-\frac{1}{k}}(\zeta)-\langle \zeta,z_t\rangle)^{-k}.
}
For $t=0$, $z_0=0$ and $\ell^{-k}$ satisfies the assumptions of \autoref{CM-thm}. Therefore, the equation
$
\sigma_k(\tau[s_0],\bar{g})=\ell^{-k}
$
has a (unique) strictly convex capillary solution. That is, $0\in \cA_k$. 

To prove that $\cA_k$ is closed, we make a minor adjustment to the final part of the closedness argument for \autoref{CM-thm}. The function $s$ that we obtained as the limit of $s_{t_i}$ is non-negative. If necessary, since $V(\widehat{\Sigma})>0$, we may translate $\widehat{\Sigma}$ to $\widehat{\Sigma} - v$ for some $v \in \operatorname{int}(\widehat{\Sigma} \cap E_{n+1}^{\bot})$ so that $s_{\Sigma - v} > 0$, while $s_{\Sigma - v}$ still satisfies the equation $\sigma_k(\tau^{\sharp}[s_{\Sigma - v}]) = \phi_{t_0}$.

\section*{Acknowledgment}
Hu was supported by National Key Research and Development Program of China 2021YFA1001800 and the Fundamental Research Funds for the Central Universities. Ivaki was supported by the Austrian Science Fund (FWF) under Project P36545.
 
We would like to thank Georg Hofst\"{a}tter, Jonas Kn\"{o}rr and Fabian Mu\ss nig for helpful discussions.

	\vspace{10mm}
	\textsc{School of Mathematical Sciences, Beihang University,\\ Beijing 100191, China,\\}
	\email{\href{mailto:huyingxiang@buaa.edu.cn}{huyingxiang@buaa.edu.cn}}
	
		\vspace{5mm}
\textsc{Institut f\"{u}r Diskrete Mathematik und Geometrie,\\ Technische Universit\"{a}t Wien,\\ Wiedner Hauptstra{\ss}e 8-10, 1040 Wien, Austria,\\} \email{\href{mailto:mohammad.ivaki@tuwien.ac.at}{mohammad.ivaki@tuwien.ac.at}}
	
	\vspace{5mm}
	\textsc{Institut f\"{u}r Mathematik, Goethe-Universit\"{a}t,\\ Robert-Mayer-Str.10, 60325 Frankfurt, Germany,\\ }
	\email{\href{mailto:scheuer@math.uni-frankfurt.de}{scheuer@math.uni-frankfurt.de}}

\begin{thebibliography}{1}

\bibitem[Ale56]{Ale56} A. D. Aleksandrov, \textit{Uniqueness theorems for surfaces in the large. I}, Vestn. Leningr. Univ. \textbf{11}(1956): 5--17.

\bibitem[Ber69]{Ber69} C. Berg, \textit{Corps convexes et potentiels sph\'{e}riques}, Det Kongelige Danske Videnskabernes Selskab Matematisk-fysiske Meddelelser,  \textbf{37}(1969), 64 pp.

\bibitem[BBCY19]{BBCY19} G. Bianchi, K. J. B{\"{o}}r{\"{o}}czky, A. Colesanti, D. Yang, \textit{The $L_p$-Minkowski problem for $-n < p < 1$},  Adv. Math. \textbf{341}(2019): 493--535.

\bibitem[BLYZ13]{BLYZ13} K. J. B{\"o}r{\"o}czky, E. Lutwak, D. Yang, G. Zhang,  \textit{The logarithmic Minkowski problem}, J. Amer. Math. Soc. \textbf{26}(2013): 831--852.

\bibitem[BG23]{BG23}  K. J. B{\"{o}}r{\"{o}}czky, P. Guan, \textit{Anisotropic flow, entropy and $L_p$-Minkowski problem}, Canad. J. Math. \textbf{77}(2025): 1--20.

\bibitem[BIS19]{BIS19} P. Bryan, M. N. Ivaki, J. Scheuer, \textit{A unified flow approach to smooth, even $L_p$-Minkowski problems}, Analysis \& PDE \textbf{12}(2019): 259--280. 

\bibitem[BIS21]{BIS21} P. Bryan, M. N. Ivaki, J. Scheuer, \textit{Parabolic approaches to curvature equations}, Nonlinear Anal. \textbf{203}(2021): 112174.

\bibitem[BIS23a]{BIS23} P. Bryan, M. N. Ivaki, J. Scheuer, \textit{Constant rank theorems for curvature problems via a viscosity approach}, Calc. Var. Partial Differ. Equ. \textbf{62}, 98 (2023).

\bibitem[BIS23b]{BIS23b} P. Bryan, M. N. Ivaki, J. Scheuer, \textit{Christoffel-Minkowski flows}, Trans. Amer. Math. Soc. \textbf{376}(2023): 2373--2393.

\bibitem[Bus12]{Bus12} H. Busemann, \textit{Convex Surfaces}, vol. 6, Courier Dover Publications, 2012.

\bibitem[CL21]{CL21} H. Chen, Q.-R.  Li, \textit{The $L_p$ dual Minkowski problem and related parabolic flows},  J. Func.  Anal. \textbf{281}(2021): 109139.

\bibitem[CY76]{CY76} S.-Y Cheng, S.-T Yau, \textit{On the regularity of the solution of the $n$-dimensional Minkowski problem}, Comm. Pure Applied Math. \textbf{29}(1976): 495--51.

\bibitem[CW00]{CW00} K.-S. Chou, X.-J. Wang, \textit{A logarithmic Gauss curvature flow and the Minkowski problem}, Ann. Inst. H. Poincar\'e Anal. Non Lin\'eaire \textbf{17}(2000): 733--751.  

\bibitem[CW06]{CW06} K.-S. Chou, X.-J. Wang, \textit{The $L_p$-Minkowski problem and the Minkowski problem in centroaffine geometry}, Adv. Math. \textbf{205}(2006): 33--83.

\bibitem[MWW25]{MWW23} X. Mei, G. Wang, L. Weng, \textit{The capillary Minkowski problem}, Adv. Math. \textbf{469} (2025): 110230.

\bibitem[MWWX24]{MWWX24} X. Mei, G. Wang, L. Weng, C. Xia, \textit{Alexandrov-Fenchel inequalities for convex hypersurfaces in the halfspace with capillary boundary, II}, arXiv:2408.13655, (2024).

\bibitem[Fir67]{Fir67} W. J. Firey, \textit{The determination of convex bodies from their mean radius of curvature functions}, Mathematika \textbf{14}(1967): 1--13.

\bibitem[Fir70]{Fir70}  W. J. Firey, \textit{Intermediate Christoffel-Minkowski problems for figures of revolution}, Israel J. Math. \textbf{8}(1970): 384--390.

\bibitem[Gho02]{Gho02} M. Ghomi, \textit{Gauss map, topology, and convexity of hypersurfaces with nonvanishing curvature}, J. Topol.  \textbf{41}(2002): 107--117.

\bibitem[GM03]{GM03} P. Guan, X.-N. Ma, \textit{The Christoffel--Minkowski problem. I. Convexity of solutions of a Hessian equation}, Invent. Math. \textbf{151}(2003): 553--577.

\bibitem[GN17]{GN17} P. Guan, L. Ni, \textit{Entropy and a convergence theorem for Gauss curvature flow in high dimension}, J. Eur. Math. Soc. \textbf{19}(2017): 3735--3761.

\bibitem[GRW15]{GRW15} P. Guan, C. Ren, Z. Wang, \textit{Global $C^2$-Estimates for Convex Solutions of Curvature Equations}, Comm. Pure Appl. Math. \textbf{68}(2015): 1287--1325.

\bibitem[GX18]{GX18} P. Guan, C. Xia, \textit{$L^p$ Christoffel-Minkowski problem: the case $1<p<k+1$},  Calc. Var. Partial Differ. Equ. \textbf{57}, 69 (2018).

\bibitem[GLW22]{GLW22} Q. Guang, Q.-R. Li, X.-J. Wang, \textit{The $L_p$-Minkowski problem with super-critical exponents}, arXiv:2203.05099, (2022).

\bibitem[HMS04]{HMS04} C. Hu, X.-N.  Ma, C. Shen, \textit{On the Christoffel-Minkowski problem of Firey's $p$-sum}, Calc. Var. Partial Differ. Equ. \textbf{21}, 137--155 (2004).

\bibitem[HLX24]{HLX24} Y. Hu, H. Li, B. Xu, \textit{The horospherical $p$-Christoffel-Minkowski and prescribed $p$-shifted Weingarten curvature problems in hyperbolic space}, arXiv:2411.17345, (2024).

\bibitem[HWYZ24]{HWYZ24} Y. Hu, Y. Wei, B. Yang, T. Zhou, \textit{A complete family of Alexandrov-Fenchel inequalities for convex capillary hypersurfaces in the halfspace}, Math. Ann. \textbf{390}(2024): 3039--3075.

\bibitem[HLYZ16]{HLYZ16} Y. Huang, E. Lutwak, D. Yang, G. Zhang, \textit{Geometric measures in the dual Brunn-Minkowski theory and their associated Minkowski problems}, Acta Math. \textbf{216}(2016): 325--388.

\bibitem[HXY21]{HXY21} Y. Huang, D. Xi, Y. Zhao, \textit{The Minkowski problem in Gaussian probability space}, Adv. Math. \textbf{385}(2021): 107769.

\bibitem[HYZ25]{HYZ25} Y. Huang, D. Yang, G. Zhang,  \textit{Minkowski problems for geometric measures}, arXiv:2502.05427, (2025).

\bibitem[Iva19]{Iva19} M. N. Ivaki, \textit{Deforming a hypersurface by principal radii of curvature and support function}, Calc. Var. Partial Differ. Equ. \textbf{58}, 1 (2019).

\bibitem[KLS25]{KLS25} W. Klingenberg, B. Lambert, J. Scheuer, \textit{A capillary problem for spacelike mean curvature flow in a cone of {M}inkowski space}, J. Evol. Equ. \textbf{25}(2025): 15.

\bibitem[Li19]{Li19}  Q.-R. Li, \textit{Infinitely many solutions for centro-affine Minkowski problem},  Int. Math. Res. Not. IMRN \textbf{2019}(2019): 5577--5596.

\bibitem[LWW20]{LWW20} Q.-R. Li, W. Sheng, X.-J. Wang, \textit{Flow by Gauss curvature to the Aleksandrov and dual Minkowski problems}, J. Eur. Math. Soc. (JEMS)  \textbf{22}(2020): 893--923.

\bibitem[LT86]{LT86} G. M. Lieberman, N. S. Trudinger, \textit{Nonlinear oblique boundary value problems for nonlinear elliptic equations}, Trans. Amer. Math. Soc. \textbf{295}(1986): 509--546.

\bibitem[LW24]{LW24} T. Luo, Y. Wei, \textit{The horospherical $p$-Christoffel-Minkowski problem in hyperbolic space}, Nonlinear Anal. TMA \textbf{257}(2025): 113799.

\bibitem[Lut93]{Lut93} E. Lutwak, \textit{The Brunn-Minkowski-Firey theory. I. Mixed volumes and the Minkowski problem}, J. Differential Geom. \textbf{38}(1993): 131--50.

\bibitem[LO95]{LO95} E. Lutwak, V. Oliker, \textit{On the regularity of solutions to a generalization of the Minkowski problem}, J. Differential Geom. \textbf{41}(1995): 227--246.

\bibitem[LXYZ24]{LXYZ24} E. Lutwak, D. Xi, D. Yang, G. Zhang, \textit{Chord measures in integral geometry and their Minkowski problems}, Comm. Pure Appl. Math. \textbf{77}(2024): 3277--3330.

\bibitem[MW23]{MW23} X. Mei, L. Weng, \textit{A constrained mean curvature type flow for capillary boundary hypersurfaces in space forms}, J. Geom. Anal. \textbf{33}(2023): 195.

\bibitem[Min97]{Min97} H. Minkowski, \textit{Allgemeine Lehrs\"{a}tze \"{u}ber die konvexen Polyeder}, Nachrichten von der Gesellschaft der Wissenschaften zu G\"{o}ttingen, Mathematisch-Physikalische Klasse, \textbf{1897}(1897): 198--219.

\bibitem[Min03]{Min03} H. Minkowski, \textit{Volumen und Oberfl\"ache},  Math. Ann. \textbf{57}(1903): 447--495.

\bibitem[Nir57]{Nir57} L. Nirenberg, \textit{The Weyl and Minkowski problems in differential geometry in the large}, Comm. Pure Appl. Math. \textbf{6}(1953) 337--394.

\bibitem[Pog52]{Pog52} A. V. Pogorelov, \textit{Regularity of a convex surface with given Gaussian curvature},  Mat. Sb. \textbf{31}(1952): 88--103 (Russian).

\bibitem[Pog71]{Pog71} A. V. Pogorelov, \textit{A regular solution of the $n$-dimensional Minkowski problem}, Dokl. Akad. Nauk. SSSR \textbf{199}(1971): 785--788; English transl., Soviet Math. Dokl. \textbf{12}(1971): 1192--1196.

\bibitem[Sch14]{Sch14} R. Schneider, \textit{Convex bodies: the Brunn-Minkowski theory}, volume 151 of Encyclopedia of Mathematics and its Applications. Cambridge University Press, Cambridge, second expanded edition, 2014.

\bibitem[STW04]{STW04} W.-M. Sheng, N. S. Trudinger, X.-J. Wang, \textit{Convex hypersurfaces of prescribed Weingarten curvatures}, Comm. Anal. Geom. \textbf{12}(2004): 213--232.

\bibitem[SW24]{SW24} C. Sinestrari, L. Weng, \textit{Hypersurfaces with capillary boundary evolving by volume preserving power mean curvature flow}, Calc. Var. Partial Differ. Equ. \textbf{63}, 237 (2024).

\bibitem[WW20]{WW20} G. Wang, L. Weng, \textit{A mean curvature type flow with capillary boundary in a unit ball}, Calc. Var. Partial. Differ. Equ. \textbf{59}(2020): 149.

\bibitem[WWX24]{WWX24} G. Wang, L. Weng, C. Xia,  \textit{Alexandrov-Fenchel inequalities for convex hypersurfaces in the halfspace with capillary boundary}, Math. Ann. \textbf{388}(2024): 2121--2154. 

\bibitem[WX19]{WX19} G. Wang, C. Xia, \textit{Uniqueness of stable capillary hypersurfaces in a ball}, Math. Ann. \textbf{374}(2019): 1845--1882.

\bibitem[Zha24]{Zha24} R. Zhang, \textit{A curvature flow approach to $L_p$ Christoffel-Minkowski problem for $1<p<k+1$}, Results in Mathematics, \textbf{79}(2024): 53.

\end{thebibliography}
\end{document}